\documentclass[12pt]{amsart}
\usepackage{latexsym,amssymb,amsmath,amsthm,graphicx,listings,comment}
\newtheorem{thm}{Theorem}
\newtheorem{propo}{proposition}

\newtheorem{coro}{Corollary}
 \newcommand{\Z}{\mathbb{Z}} 
\def\X{\mathcal{X}} \def\O{\mathcal{O}} \def\A{\mathcal{A}} \def\O{\mathcal{O}} 
\let\paragraph\subsection

\title{On Fredholm determinants in topology}

\author{Oliver Knill}
\date{December 23, 2016}
\address{Department of Mathematics \\ Harvard University \\ Cambridge, MA, 02138, USA }

\begin{document}

\maketitle

\begin{abstract}
For a finite simple graph $G$ with adjacency matrix $A$, the
Fredholm determinant $\zeta(G)={\rm det}(1+A)$ is 
$1/\zeta_G(-1)$ for the Bowen-Lanford zeta function
$\zeta_G(z)={\rm det}(1-z A)^{-1}$ of the graph. The connection graph $G$' of
$G$ is a new graph which has as vertices the set $V'$ of all complete subgraphs of $G$ and where two
such complete subgraphs are connected, if they have a non-empty intersection.
More generally, the connection graph of an abstract finite simplicial complex
or even CW complex $G$ has as vertices the simplices or cells in $G$, where two are connected 
if they intersect. We prove that for any $G$, the Fredholm characteristic $\psi(G)=\zeta(G')$  
is equal to the Fermi characteristic $\phi(G) = (-1)^{f(G)}$, where $f(G)$ is the number of odd
dimensional cells in $G$; the functional $f(G)$ is a valuation for which
Poincar\'e-Hopf and Gauss-Bonnet formulas hold.
Given $\omega(x)=(-1)^{{\rm dim}(x)}$, we can 
see the Fermi characteristic $\phi(G)=\prod_x \omega(x)$
as a cousin of the Euler characteristic $\chi(G)=\sum_x \omega(x)$ which sums the signatures of simplices.
The main result is the unimodularity theorem $\psi(G)=\phi(G)$ which relates an algebraic and a combinatorial quantity.
We illustrate this with prime graphs, where $\omega(x)=-\mu(x)$ is the M\"obius function of an integer. 
A key proposition for the proof of the theorem is that if $i(x) = 1-\chi(S(x))$ is the Poincar\'e-Hopf index
of $x$, where $S(x)$ is the unit sphere of $x$, then $\psi(G \cup \{x\}) = i(x) \psi(G)$. If $S(x)$ is a graph theoretical 
sphere, then $i(x) \in \{-1,1\}$ proving inductively that $\psi$ is $\{-1,1\}$-valued. The
unimodularity theorem follows then by induction by building up the simplicial complex
cell by cell, using that spheres have Euler characteristic $0$ or $2$. 
Having established that the Fredholm matrix $1+A(G')$ of 
a simplicial complex $G$ is unimodular, the entries of the 
Green function $g_{ij}=[(1+A(G'))^{-1}]_{ij}$ are always integers.
They appear of interest as experiments indicate that the range of $g$ is
a combinatorial invariant for $G$: we conjecture that global or local
Barycentric refinements of simplicial complexes $G$ do not change the range of $g$. 
\end{abstract}

\section{The Fredholm determinant of a graph}

Fredholm matrices appear naturally in graph theory. They arise most prominently in the 
{\bf Chebotarev-Shamis forest theorem} \cite{ChebotarevShamis1,ChebotarevShamis2}
which tells that ${\rm det}(1+L)$ is the number of rooted forests in a graph $G$, 
if $L$ is the Kirchhoff Laplacian of $G$. This forest theorem follows readily from
the {\bf generalized Cauchy-Binet formula} \cite{CauchyBinetKnill}
${\rm det}(1+F^T G) = \sum_P {\rm det}(F_P) {\rm det}(G_P)$
which holds for any pair $F,G$ of $n \times m$ matrices and were the right 
hand side is a dot product of the minor vector giving all possible minors of the
matrix, defined by the index set $P$ which can be empty in which case ${\rm det}(A_P)=1$.
The forest theorem uses this with $L = d^T d$, where $d={\rm grad}$ and $d^T={\rm div}$ are
Poincar\'e's incidence matrices. The Fredholm determinant
${\rm det}(1+L)$ is then $\sum_P {\rm det}(d_P)^2$, which directly
counts rooted forests. \\

Similarly, any Fredholm determinant ${\rm det}(1+A)$ of a matrix $A$ can
be written as ${\rm det}(1+A)=\sum_P {\rm det}(A_P) {\rm det}(1_P)$ which implies the
well known formula ${\rm det}(1+A)=\sum_{k=0}^{\infty} {\rm tr}(\Lambda^k(A))$ 
where $\Lambda^k(A)$ is the $k$'th exterior power of $A$. That expansion is at the 
heart of extending determinants to Fredholm determinants in infinite dimensions, in particular
if $A$ is trace class \cite{SimonTrace}. In this paper, we look at the adjacency matrix $A$, which 
unlike the Laplacian $L$ of a graph is not positive semi-definite so that one can not 
write $A=F^T F$ for some other matrix. Indeed, the determinant of an adjacency matrix $A(G)$ 
of a graph $G$ appears rather arbitrary in the 
strong sense that experiments show that a limiting distribution emerges 
when looking at the random variables $X(G)={\rm det}(A(G))$ on probability spaces of graphs.
See Figure~(\ref{distribution}). \\

Given a finite simple graph $G$ with adjacency matrix $A$, we call 
$1+A$ the {\bf Fredholm adjacency matrix} and $\zeta(G) = {\rm det}(1+A)$ the {\bf Fredholm determinant}
of $G$. The Fredholm determinant of a graph can be pretty arbitrary as the following 
examples show: for complete graphs $G=K_d$, where the
eigenvalues of the adjacency matrix are $-1$ with multiplicity $n$ and $n$ with multiplicity $1$
and the Fredholm determinant is $0$ for $d>1$. 
For cyclic graphs $C_n$, the Fredholm determinant is $6$-periodic
in $n$, zero for $n=6k$, and $3$ for $n=2k$ not divisible by $3$ and $-3$ for $n=2k+1$ not divisible
by $6$. For wheel graphs $W_n$ with $n+1$ vertices, it is $(n-3) (-1)^n$ if $n$ is not divisible by
$3$ and $0$ else. Measuring the statistical distribution of Fredholm determinant on classes of 
random graphs suggests that the normalized
distribution of the adjacency or Fredholm determinants produces an absolutely continuous limit with 
singularity. Also this appears not yet explored theoretically but it illustrates that the 
Fredholm determinant of a general graph can be pretty arbitrary. \\

We can think of ${\rm det}(A)$ as a partition function or a ``path integral",
in which the underlying paths are fixed-point-free permutations of the vertices of the graph. 
The determinant generates therefore {\bf derangements} and 
$x \to \pi(x) \to \pi(\pi(x)) \dots$ defines the oriented paths. On the other hand,
the Fredholm determinant $\zeta(G) = {\rm det}(1+A)$ is a partition function for all oriented 
paths in the graph as $x \to \pi(x)$ can now also have pairs $(a,b) \in E$ as transpositions and
vertices $v \in v$ as fixed points. One can see the effect of changing from determinants to Fredholm determinants
well when replacing the determinants with the {\bf permanent}, the Bosonic analogue 
of the determinant: ${\rm per}(A)$ is the number {\bf derangements} of the vertex set of 
the graph while ${\rm per}(1+A)$ is the number of all 
{\bf permutations} of the vertex set honoring the connections. For the complete graphs $G=K_n$ in particular,
${\rm per}(1+A(K_n))$ generates the {\bf permutation sequence} $1,2,6,24,120,720, \dots$ 
while the permanent of the adjacency matrix ${\rm per}(A(K_n))$ generates 
the {\bf derangement sequence} $0,1,2,9,44,265, \dots$. It is therefore not 
surprising that Fredholm determinants are natural. \\

The distribution of Fredholm determinants
changes drastically if we evaluate them on the set of {\bf connection graphs},
graphs which have the set of simplices of a graph $G$ as subgraphs and where
two simplices are connected if they intersect. Connection graphs are also defined
for abstract finite simplicial complexes or even finite CW complexes. If we talk about a graph $G$, we usually 
understand it equipped with the Whitney complex, the set of complete subgraphs of $G$. 
But any simplicial complex structure or CW complex structure on the graph works. It does not even
have to come from simplices. The {\bf graphic matroid} of a graph is an example where the connection graph 
has forests in $G$ as vertices and has two forests connected if some trees in it share a common branch.  \\
 
For the smaller {\bf Barycentric refinement} $G_1$ of $G$, two simplices are connected only if and only
if one is contained in the other. The graph $G_1$ has the same vertices than $G'$ but is a subgraph of $G'$.  
Connection graphs are in general much higher dimensional than the graph $G$ or even the Barycentric refinement $G_1$:
for a triangle $G=K_3$ already, $G'$ a graph which contains the complete graph $K_4$. 
Small spheres like the octahedron are examples where $G$ and $G'$ are not homotopic because $G'$ 
has Euler characteristic $0$ while the octahedron, as a 2-sphere has Euler characteristic $2$. 
However this only is the case because the sphere is too small. For the icosahedron $G$ already, the connection graph  $G'$
a two sphere again. In general, the Barycentric refinement $G_1$ of a graph $G$ has a connection graph $G_1'$ 
which is homotopic to $G_1$ and so to $G$. 
This implies then that all cohomology groups of $G$ and $G_1'$ agree. While $G$ and $G_1$ are not 
homeomorphic as already their dimension is different in general, they can be useful in geometry, 
like Barycentric refinements.
They can be used for example to {\bf regularize singularities} as they "homotopically fatten" the 
`` discrete manifolds" or simplicial complexes and still have the same homotopy type 
after applying one Barycentric refinement.  \\

We got interested in connection graphs in the context of 
``connection calculus", a calculus where differential forms are not functions on simplices 
but on pairs or $k$-tuples of connecting simplices in the simplicial complex. 
The corresponding cohomology is compatible with calculus in the sense that common theorems
like Gauss-Bonnet \cite{cherngaussbonnet}, Poincar\'e-Hopf \cite{poincarehopf}, 
Euler-Poincar\'e or Kuenneth \cite{KnillKuenneth} or Brouwer-Lefschetz fixed point theorem \cite{brouwergraph}
generalize when Euler characteristic is replaced by {\bf Wu characteristic} but for which the cohomology is finer.
The cohomology already allows to distinguish spaces which classical simplicial cohomology can not, like the 
M\"obius strip and the cylinder \cite{CaseStudy2016}. \\

The Bowen-Lanford {\bf zeta function} \cite{BowenLanford}
of a graph $G$ with adjacency matrix $A$ 
is defined as the complex function $\zeta(z) = 1/{\rm det}(1-z A)$, from $\mathbb{C}$
to $\mathbb{C}$, where $A$ is the adjacency matrix of $A$. If $r$ is the spectral radius of $A$,
the absolute value of the largest eigenvalue of $A$, then the function $\zeta$ is analytic in $|z|<1/r$.
The Fredholm determinant of $A$ is then $1/\zeta(-1)$, which if $-1$ is an eigenvalue of $A$
is defined as $0$. Zeta functions are of interest as they relate with topology.
We have a Taylor expansion $\zeta(z) = \exp(\sum_{k=1}^{\infty} (N_k/k) z^k)$ for small $|z|$, where $N_k$ is 
the number of rooted closed paths of length $k$ in the graph. The zeta function is therefore a generating
function for a dynamical property of the graph, the dynamical system being the Markov chain defined
by the graph.  It is in particular an {\bf Artin-Mazur zeta function} and 
a special case of the {\bf Ruelle zeta function} \cite{ruellezeta}. 
Since $\zeta(z)$ is a rational function, the sum can be understood for general $z$ only 
by analytic continuation. While for general graphs, $\zeta(-1)$ can be quite
arbitrary, we will see that for connection graphs $G'$, the analytic continuation of the divergent
series $\sum_{k=1}^{\infty} (N_k/k) z^k$ at $z=-1$ is either $0$ or $\pi i$
and that the case $0$ appears if and only if there is an even number of odd-dimensional complete
subgraphs of the original graph $G$. If $v_k(G)$ is the number of $k$-dimensional simplices, then 
the Euler characteristic $\chi(G) = \sum_{k=0}^{\infty} (-1)^k v_k(G)$ is the difference 
of $b(G)-f(G)$, where $b(G)= \sum_{k=0}^{\infty} v_{2k}(G)$ and 
$f(G) = \sum_{k=0}^{\infty} v_{2k+1}(G)$.  \\

While these cardinalities $\psi(G)$ appear
naturally when evaluating the zeta function of a connection graph at $z=-1$, 
concrete examples of zeta function of a connection graph of some of the simplest
graphs indicate, that the result $\zeta_{G'}(-1) \in \{-1,1\}$ is not that obvious:
$\zeta_{K_1'}(z) = -1/z$, $\zeta_{K_2'}(z_) = -1/(z^3-2z)$, 
$\zeta_{K_3}'(z) = 1/(-z^7+15 z^5+26 z^4-3 z^3-24 z^2-2 z+6)$. 
$\zeta_{C_5}'(z) = 1/(z^{10}-15 z^8-10 z^7+70 z^6+78 z^5-100 z^4-160 z^3-15 z^2+30 z-4)$.
These functions evaluated at $-1$ either have the value $1$ or $-1$. 
By the way, the topic of Fredholm determinants has appeared in
the movie "Good will hunting" as one of the problems involves the Bowen-Lanford function: 
the last blackboard problem in that movie asks for the generating function for walks from a 
vertex $i$ to a vertex $j$ in concrete graph. The answer 
is $[(1-z A)^{-1}]_{ij}$ which by Cramer is expressed by the {\bf adjugate matrix} as
the rational function ${\rm det}(1-z A(j,i))/{\rm det}(1-z A) 
= \zeta(z) {\rm det}(1-z (-1)^{i+j} A(i,j))$,
where $A(i,j)$ is the matrix obtained by deleting row $i$ and column $j$ in $A$. \\

In search of a prove of the theorem, it can be helpful to see the connection graph as a geometric space and see
a permutation $\pi$ of its vertices as a one-dimensional oriented ``submanifold", a collection
of disjoint cyclic oriented paths or ``strings". Since a permutation compatible with the graph as a 
``flow" on the geometry, the Fredholm determinant sums over all possible ``measurable 
continuous dynamical systems $T$". They can be considered flows in $G'$ in the sense that for every vertex $x$,
the pair $(x,T(x))$ is an edge in $G'$. We call them measurable because $T$ is not continuous 
in the geodesic distance metric of $G'$. \\

The signature $\omega(\pi)$ of a flow is the product of the signatures of the individual 
connected cyclic components of the flow=permutation $\pi$. The Fredholm determinant 
$\psi(G)={\rm det}(1+A')$ of the adjacency matrix $A'$ of the connection graph $G'$ is then a 
{\bf path integral} $\psi(G)=\sum_\pi \omega(\pi)$, where $\pi$
runs over all possible flows in $G'$. The unimodularity theorem tells then 
that the Fredholm determinant $\psi(G)$ is equal to the {\bf Fermi characteristic} 
$\phi(G) = \prod_x \omega(x)$, where $x$ runs over all 
complete subgraphs of $G$, showing so that $\psi$ is a {\bf multiplicative valuation}
$\psi(G \cup F) = \psi(G) \psi(H)/\psi(G \cap H)$. We can compare $\psi(G)$ with the additive valuation
$\chi(G)$ on graphs which is the {\bf Euler characteristic} $\chi(G)=\sum_x \omega(x)$ or with
the {\bf Wu characteristic} $\omega(G) = \sum_{x \sim y} \omega(x) \omega(y)$, summing over all edges $(x,y)$
of the connection graph $G'$ \cite{Wu1953,valuation}. But unlike Euler characteristic $\chi$ or 
Wu characteristic $\omega$, the functional $\psi$ is not a combinatorial invariant, as $\psi$ 
is constant $1$ on Barycentric refinements.  \\

Functionals like the range of the unimodular Green function $g_{ij} [(1+A')^{-1}]_{ij}$
values appear to be combinatorial invariants - at least in 
experiments. We have not proven this observation yet. The closest analogy which
comes to mind is an invariant found by Bott \cite{Bott52} who coined the term 
{\bf combinatorial invariant} as a quantity which is invariant under Barycentric subdivision.
By Cramer, $g_{ii} \psi(G)$ is the Fredholm characteristic of the geometry in which cell $i$ 
is removed and $g_{ij} \psi(G)$ a Fredholm characteristic of a geometry, where outgoing 
connections from cell $i$ and incoming connections to cell $j$ are snapped. 

\section{Connection graphs}

If $G=(V,E)$ is a finite simple graph, we denote by $V_1$ the set of all 
complete subgraphs of $G$. Also named {\bf simplices} or {\bf cliques}, these subgraphs 
are points of the {\bf Barycentric refinement} $G_1$ of $G$, which has
as a vertex set $V_1$ the set simplices and where two such simplices are connected if 
one is contained in the other. The larger {\bf connection graph} $G'$ has the same vertex
set like $G_1$. In that graph, two simplices are connected, if they have a non-empty 
intersection. Unlike for $G_1$, for which the maximal dimension of $G$ and $G_1$ 
are the same, the graph $G'$ is in general ``fatter": for a one-dimensional circular graph for example, 
the graph $G'$ has triangles attached to each edge. More generally, $G'$ contains
complete subgraphs $K_{n+1}$ if there is a vertex of $x$ which is contained in $n$ simplices:
the unit ball of a simplex $x'=(x)$ belonging to an original vertex is a complete graph. \\

Also if we primarily want to analyze the graph case, it is convenient to look at more
general simplicial complex structures on the graph. 
Assume $G$ is an {\bf abstract finite simplicial complex}, a finite set $V$ equipped with
a collection $V'$ of finite subsets of $V$ such that for every $A \in V'$ and every subset $B$ 
of $A$ also $B \in V'$, then the connection graph $G'=(V',E')$ is the finite simple graph for
which two elements in $V'$ are connected, if they intersect. An abstract
simplicial complex is not only a generalization of a graph, we can see it as a
{\bf structure} imposed on a graph similarly as a {\bf topology}, an {\bf order structure} 
or {\bf $\sigma$-algebra} is imposed on a set. Much of the graph theory literature sees a graph $G=(V,E)$ 
by default equipped with the one-dimensional {\bf skeleton complex} $V \cup E$. The largest complex is
the {\bf Whitney complex} on $G$, which is the set of all complete subgraphs. 
Graphs can handle many simplicial complexes as 
given an abstract finite simplicial complex $G$, the Barycentric refinement $G_1=(V_1,E_1)$
is a graph which has as vertex set $V_1$ the set of elements in $G$ and has 
$E_1=\{ (a,b) \; | \; a \subset b \; {\rm or} \; b \subset a \}$. \\
Given an abstract finite simplicial complex $G$ given by a finite set $V$ equipped with a collection
$V'$ of finite subsets, the {\bf connection graph} of $G$ is the graph with vertex set $V'$, where
two vertices $x,y$ are connected, if they intersect as subsets of $V$. \\

More general than simplicial complexes are discrete {\bf CW complexes}. This structure
is built up inductively. It recursively defines also the notion of contractibility and a notion 
of sphere in this structure. This generalizes the Evako setup in the graph case (see \cite{KnillJordan})
Start with the empty set, which is declared to be the $(-1)$-sphere. It does not contain
any cells. A {\bf CW-complex} is declared to be a $d$-sphere if when punctured becomes contractible and 
which has the property that every unit sphere $S(x)$ of a $(d-1)$ sphere. 
The {\bf unit sphere} of a cell $x$ is the CW-sub complex of $G$ containing all cells which are 
either part of $x$ or which contain $x$. Also inductively, a CW complex $G$ is {\bf contractible} 
if there exists a cell $x$ such that both $S(x)$ and $G$ without $x$ are contractible. 
Inductively, if $G$ is a CW-complex one can build a larger complex by choosing a sphere $H$ in $G$,
then do an extension over $H$ with a new cell $x$, producing so a larger CW complex. 
The unit ball of $x$ has $S(x)=H$ as a boundary and the new cell is declared to be $1+{\rm dim}(H)$.
Starting with the empty set, one can build up like this structures which are more 
general than finite simplicial complexes but which still do not
(unlike the classical definition of CW complexes) invoke the infinity axiom in Zermelo-Frenkel.
The {\bf connection graph} of a CW complex is the finite simple graph $(V,E)$, where
$V$ is the set of cells and where two cells are connected if they intersect. The Fermi characteristic
and Fredholm characteristic of a CW complex are defined in the same way as before.  \\

{\bf Remarks.}  \\
{\bf 1)} If one looks at the unit balls of cells as a ``cover" of the CW complex, then
the connection graph plays the role of the {\bf nerve graph} in \v{C}ech setups. The elements in the
cover have however more structure because their boundaries are always graph theoretical spheres. 
Every simplicial subcomplex of the Whitney graph is a CW complex in the just given sense and any 
finite classical CW complex can be described combinatorially as such. There is a more general 
notion of {\bf abstract polytope} given as a poset satisfying some axioms but the 
unimodularity theorem won't generalize to that.
The {\bf Barycentric refinement} of a CW complex is a graph containing the cells as vertices and 
where two cells are connected if one is contained in the other. This graph is equipped with the
Whitney complex structure which is again a CW complex, but which is much larger. But again, like for
the notion of simplicial complex, we see that it can be implemented as a graph and that we can see
therefore a CW complex as a {\bf structure imposed on a graph}. There is
still a reason to keep the notion of CW complexes: the connection graph of a CW complex has the
same number of vertices than cells and the unimodularity theorem applies to it.
The connection graph of the graph attached to the CW complex would be
much larger. This will be relevant when we look at prime graphs. It is also relevant in general:
if we look at a Barycentric refinement of a CW complex, then it is a graph which has the same number
of vertices than the connection graph of this CW complex. The prime graph and prime connection graph
considered below are then a Morse filtration of the  Barycentric refinement and Connection graphs of the
simplest simplicial complex one can imagine: the set of primes equipped with the set of all subsets 
as complex. \\

{\bf 2)} As the Barycentric refinement of a simplicial complex or CW complex encodes
most essential topological features in $G$, there is not much loss of generality by looking at
graphs rather than simplicial complexes. Still, the slightly increased generality can make
the result more transparent. But applying it to graphs is more intuitive. The structure of a CW 
complex is very natural and practical: look at the cube graph for example. Since the unit spheres are 
graphs without edges, it is a one-dimensional graph when equipped with the Whitney complex. If we
stellate the 6 faces, then we have a larger graph with 14 edges, the stellated cube. Its Whitney complex
is large as there are already 24 triangles present. But adding 6 cells with $C_4$ boundaries, we get the
familiar picture of a cube with f=6 two dimensional faces, e=12 one dimensional edges and v=8 vertices. 
This is how already Descartes counted the Euler characteristic $v-e+f=2$ \cite{Aczel}. 
History shows \cite{lakatos} how difficult it
has been to get to a good notion of "polyhedron" (see also \cite{Gruenbaum2003}) and one usually refers to 
Euclidean embeddings, using through notions like convexity to define it properly \cite{gruenbaum}. \\

{\bf 3)} The notion 
of CW complex in the discrete allows (without using any Euclidean notions) to give a decent definition
of polyhedron as a CW complex which is a sphere in the sense that removing one cell renders the CW complex
contractible and that every unit sphere of any cell is a sphere. The essential foundation to that
is in Whitehead \cite{Whitehead} already in the 1930ies. What was new in the 90ies is the realization
that one can do all this on graphs without using the continuum. 
There are three reasons why the language of finite graphs is more convenient:
it is an intuitive structure which small kids can grasp already; it is a data structure which exists
in many higher level programming languages. All results discussed here can be explored with a few lines of
code (provided below). Finally, it is a finite structure; finite mathematics works also in a
framework of finitists like Brouwer or strict finitist, which many computer scientists are, wanting to
implement the complete structure faithfully.  \\

{\bf 4)} Whitehead {\bf homotopy} has been ported to discrete structures by first 
defining {\bf contractibility} inductively: it is either the $1$ vertex graph or a graph 
for which there exists a vertex $x$ such that both the unit sphere $S(x)$ is contractible
and such that the graph without $x$ is contractible. Contractible graphs have Euler characteristic
$1$. A {\bf homotopy step} is the process of an addition or removal of a vertex $x$,
for which $S(x)$ is contractible. 
Two graphs are homotopic, if one can get one from the other by applying a sequence of
homotopy steps. A graph can be homotopic to a 1-point graph without being contractible.
An example is the {\bf dunce hat} which shows that one first has to enlarge the graph before it becomes
contractible. Homotopic graphs have the same cohomology and Euler characteristic however. \\
While the Barycentric refinement $G_1$ is homotopic to $G$, $G'$ is not homotopic to $G$ 
in general but this happens only if {\bf very small} homotopically non-trivial spheres are present.
For the octahedron $G$ for example, the graph $G'$ has Euler characteristic $0$ while $G$ 
has Euler characteristic $2$ so that $G$ and $G'$ can not be homotopic. 
This only happened because the geometry was too small and the $H^2$ cohomology
of $G$ collapsed in $G'$. For smooth enough graphs, $G'$ is homotopic 
to $G$ it is homotopic and has the same cohomology. The connection graph of the Barycentric 
refinement $G_2$ of $G_1$ for example is always homotopic to $G_1$ and so to $G$ as loosing the
additional connection bonds does allows to morph from $G_1'$ to $G_2$ without changing
the topology. \\

{\bf 5)} The connection graphs emerged for us in the context of the {\bf Wu characteristic}
$$ \omega(G) = \sum_{(x,y) \in E'} (-1)^{{\rm dim}(x) + {\rm dim}(y)} \;  $$
of a graph $G$ which is a ``second order" Euler characteristic. The Wu characteristic
shares all important properties of Euler characteristic: it is multiplicative and additive
with respect to multiplication or addition of the geometric structures;
there is a compatible calculus, cohomology and theorems
like Gauss-Bonnet, Poincar\'e-Hopf, Euler-Poincar\'e, Kuenneth or Lefschetz generalize.
A relation of the Fredholm adjacency matrix and the Wu characteristic is given by
$$  \omega(G) = {\rm tr}( (1+A') J) \; , $$
where $J$ is the {\bf checkerboard matrix} $J_{ij} = (-1)^{i+j}$.

\begin{figure}[!htpb]
\scalebox{0.3}{\includegraphics{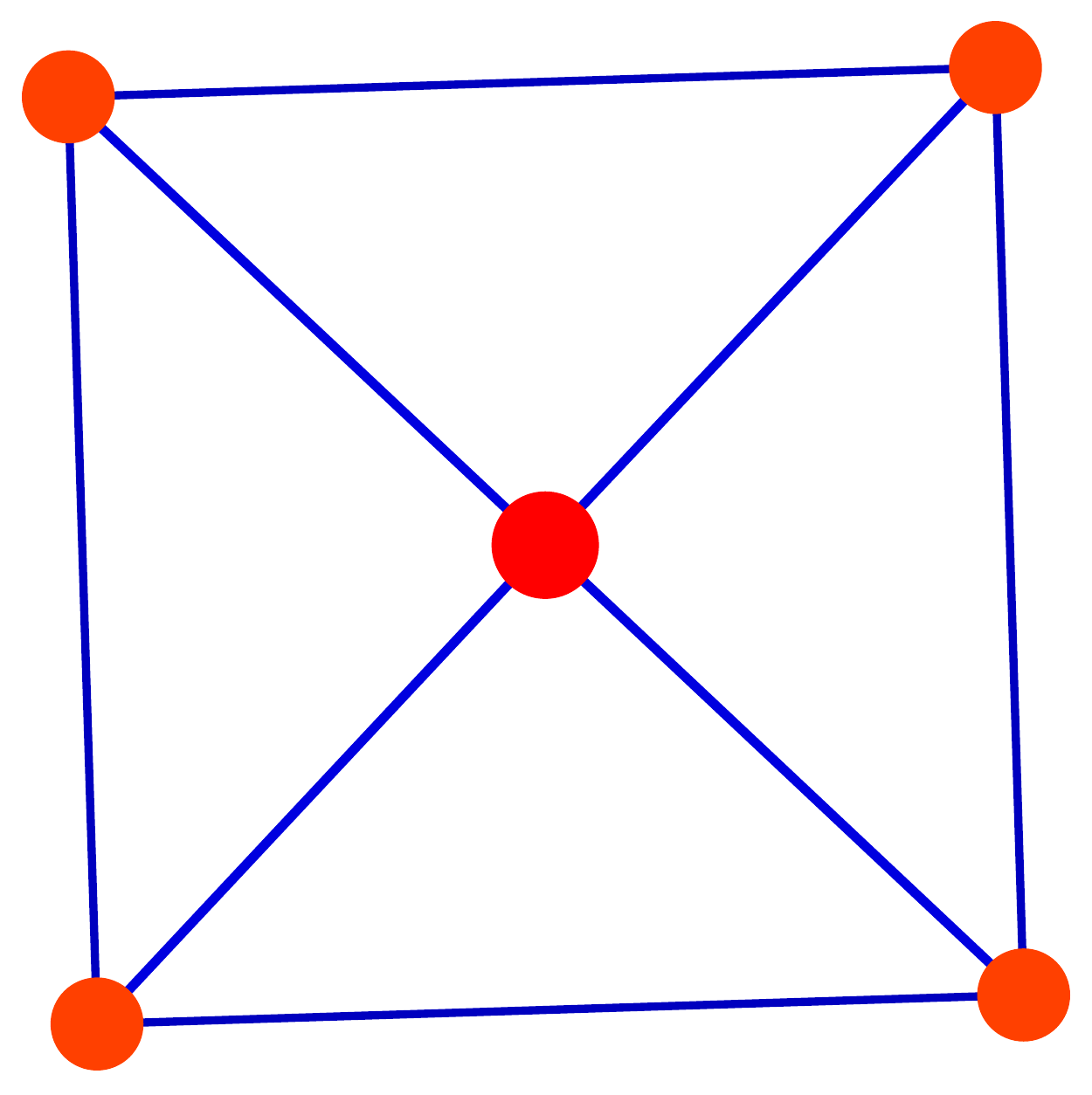}}
\scalebox{0.3}{\includegraphics{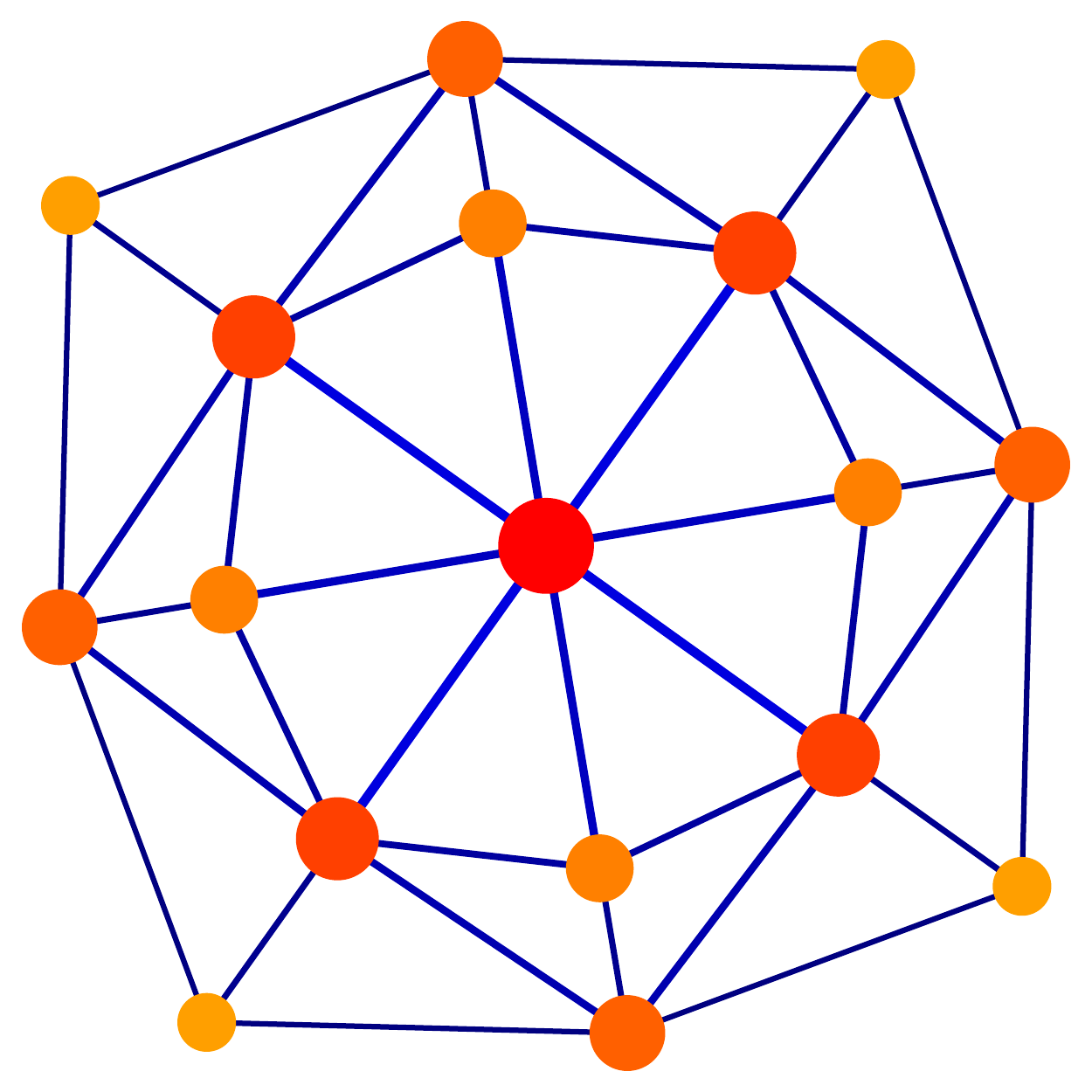}}
\scalebox{0.3}{\includegraphics{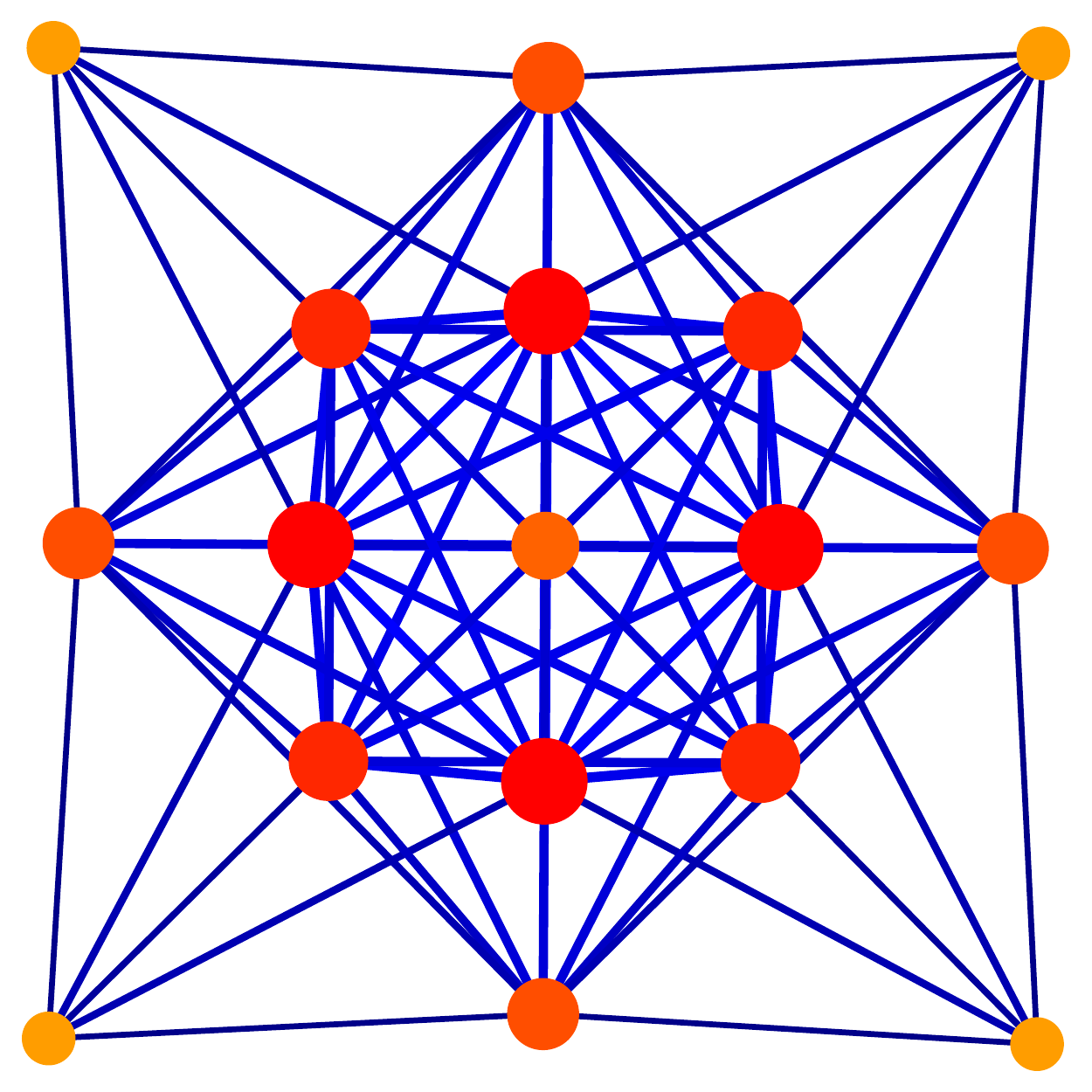}}
\caption{
The wheel graph $G$, its Barycentric refinement $G_1$ 
and the connection graph $G'$. 
\label{figure1}
}
\end{figure}

\begin{figure}[!htpb]
\scalebox{0.4}{\includegraphics{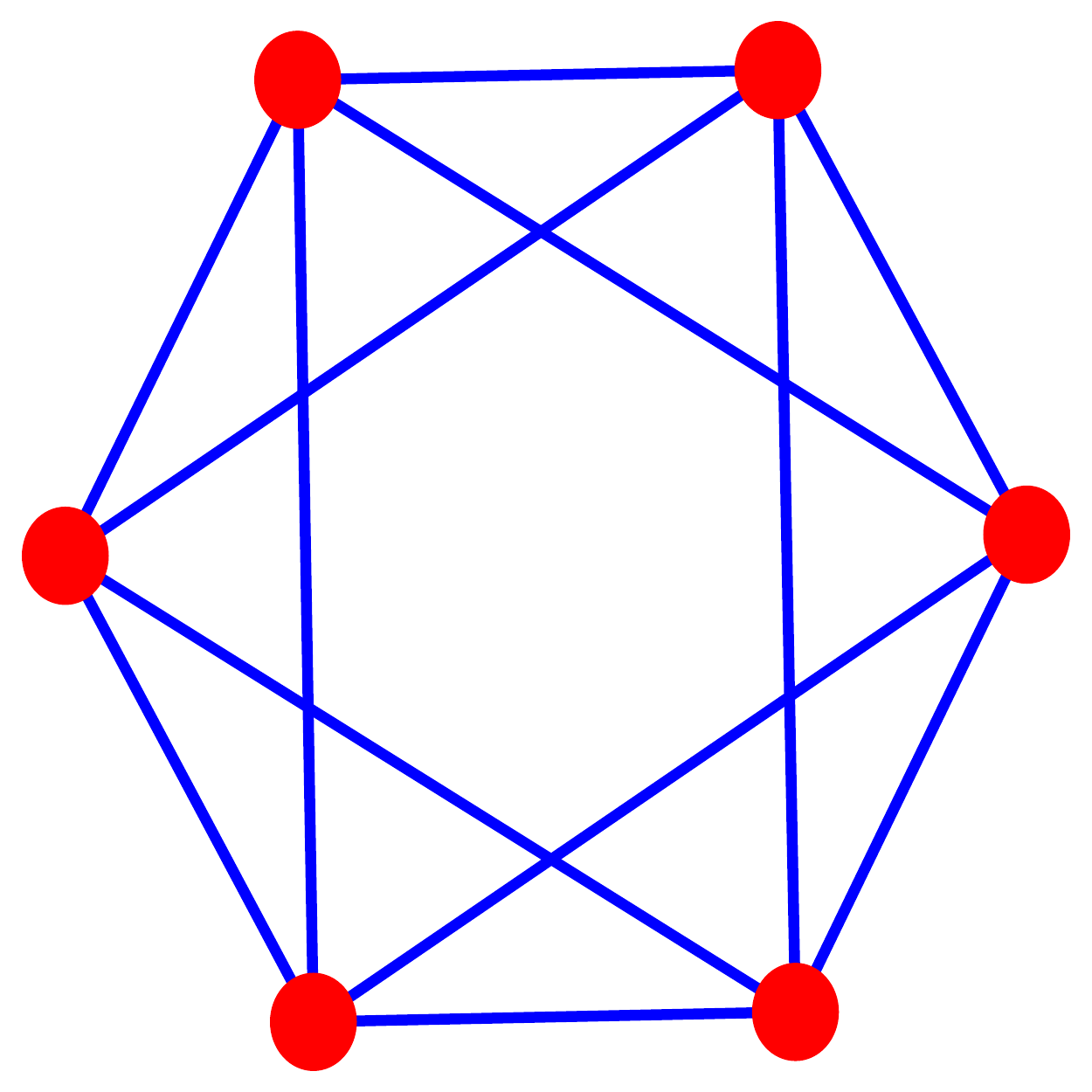}}
\scalebox{0.4}{\includegraphics{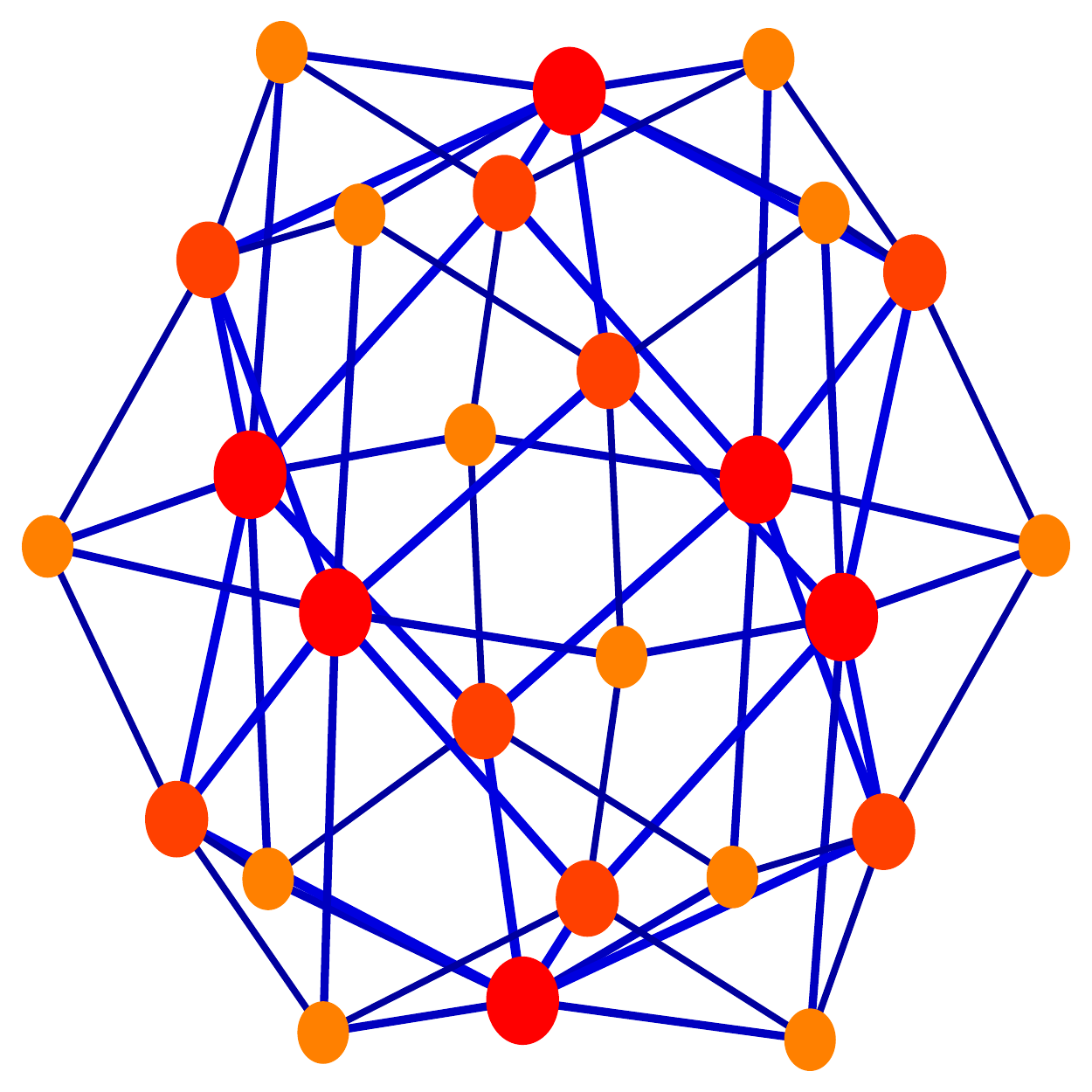}}
\caption{
The octahedron graph $G$ is an example of a discrete 2-sphere.
Unlike $G$ or its Barycentric refinement $G_1$, the 
connection graph $G'$ is now contractible.
The graph was so small that the connections filled up the 2-sphere. 
\label{figure2}
}
\end{figure}

\begin{figure}[!htpb]
\scalebox{0.4}{\includegraphics{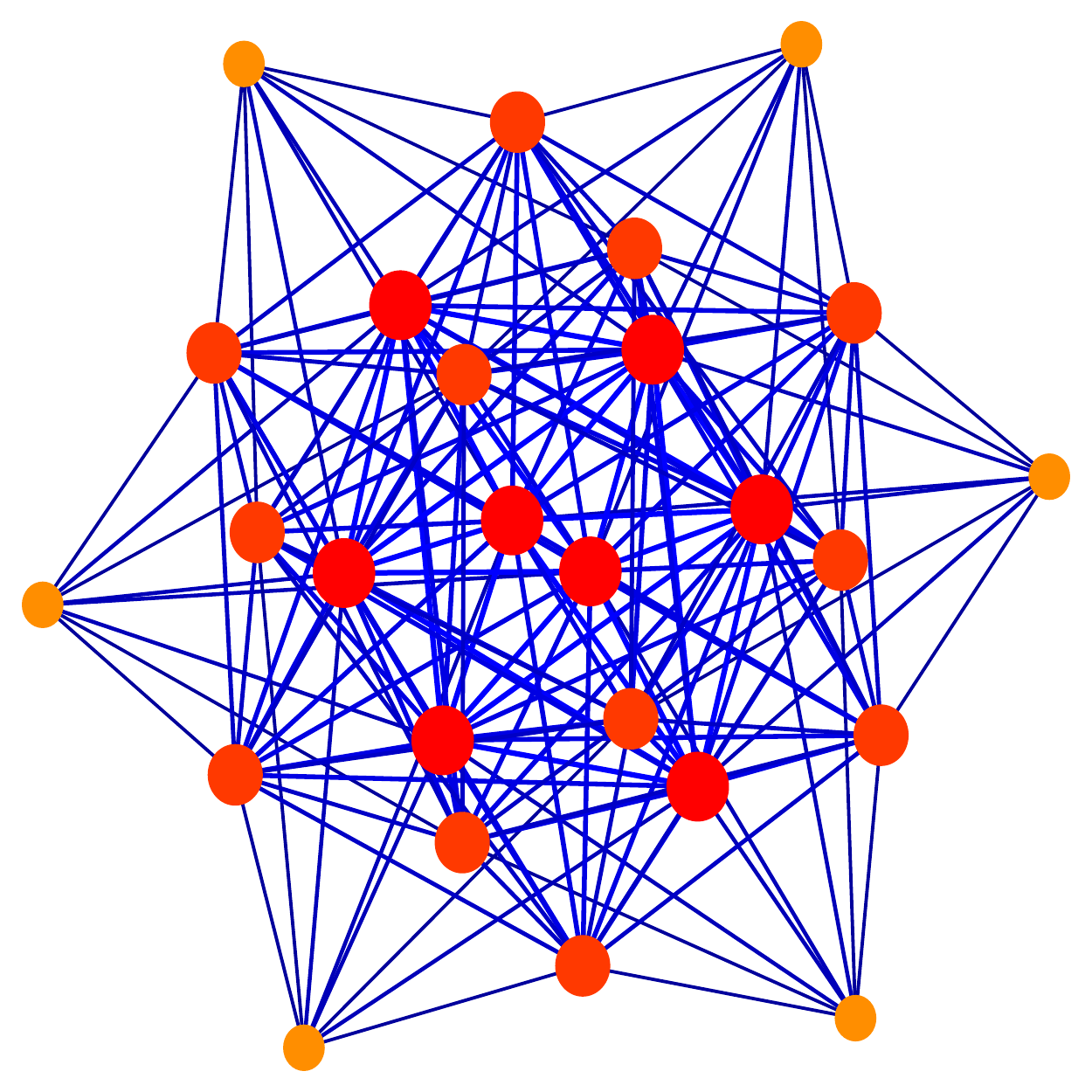}}
\scalebox{0.4}{\includegraphics{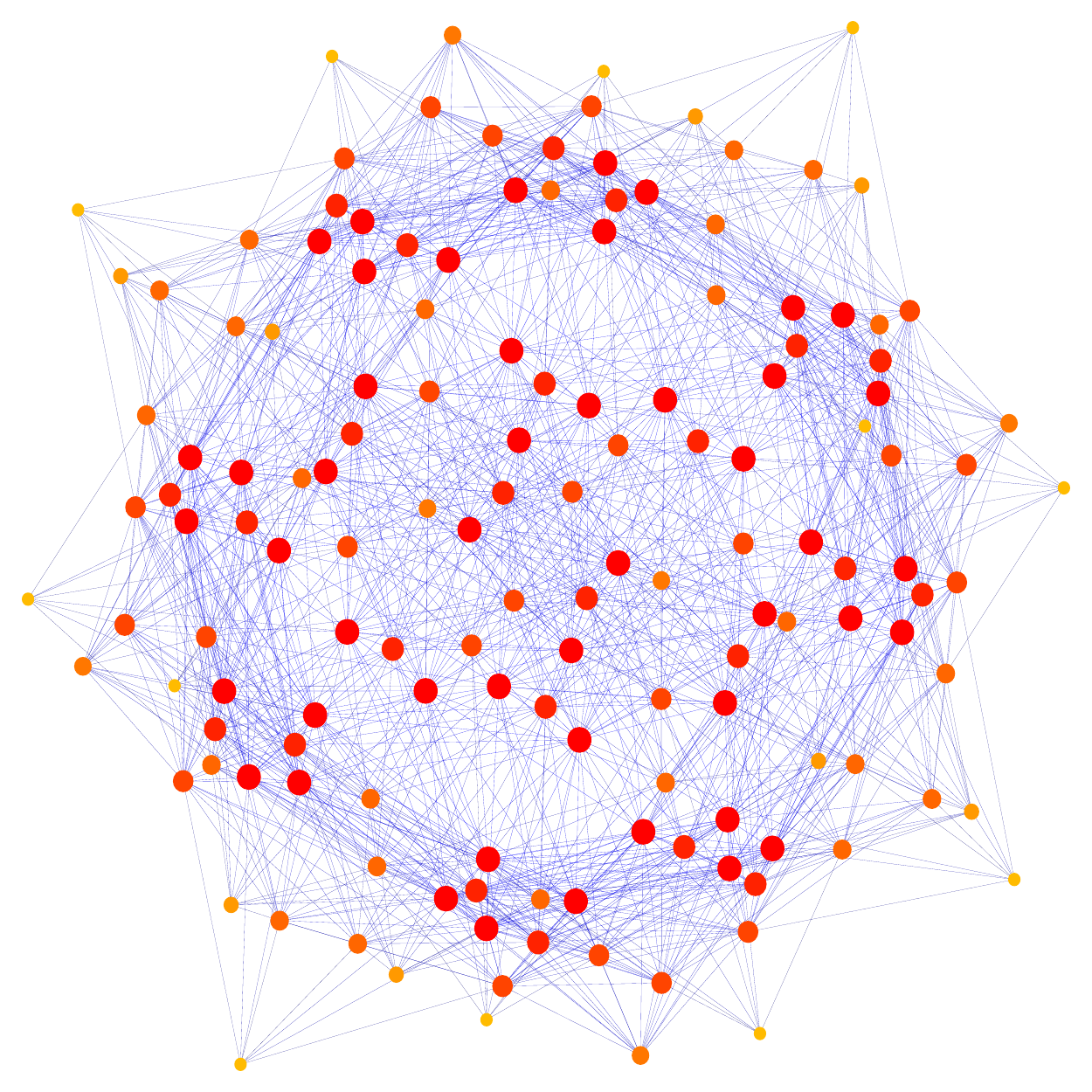}}
\caption{
The Barycentric refinement $G_1$ of the octahedron is again a 2-sphere.
Its connection graph $G'_1$ is now homotopic to $G$. 
\label{figure3}
}
\end{figure}

\section{The Unimodularity theorem}

An integer matrix $M$ is called {\bf unimodular} if its determinant is either $1$ or $-1$. 
By the explicit {\bf Cramer-Laplace inversion formula}, unimodularity 
is equivalent to the fact that its inverse $M^{-1}$ is an 
integer-valued matrix. Because an unimodular matrix $M$ 
is non-singular in the ring $M(n,Z)$ of integer matrices, a unimodular matrix $M$ is an element 
in $GL(n,Z)$.  \\

Let $G$ be a graph equipped with a simplicial complex or a CW-complex. 
With the Fermi number $f(G)$ giving the number of odd dimensional simplices of cells
in $G$, the Fermi characteristic is defined 
as $\phi(G) = (-1)^{f(G)} =\prod_x \omega(x)=\prod_x (-1)^{{\rm dim}(x)}$. The 
{\bf Fredholm characteristic} 
$$   \psi(G) = {\rm det}(1+A(G')) $$ 
is the Fredholm determinant of the adjacency matrix of the connection graph $G'$
of the complex. 

\begin{thm}[Unimodularity]
For any graph $G$ equipped with a simplicial or CW structure: 
$$  \psi(G) = \phi(G) \; . $$
\end{thm}

The following corollary of the theorem is actually equivalent to the theorem, at least
on simplicial complexes. As usual,
if $G=(V,E)$ and $H=(W,F)$ are two finite simple graphs, the {\bf intersection graph}
$G \cap H = (V \cap W, E \cap F)$ and {\bf union graph} 
$G \cup H = (V \cup W, E \cup F)$ are both finite simple graphs. 
Also in the slightly more general case, if $G,H$ are finite abstract simplicial complexes
or CW complexes, one can look at the intersection $G \cap H$ and union complex $G \cup H$. 

\begin{coro}
For any finite simple graphs $G,H$, simplicial complexes or CW complexes, 
the determinant formula
\begin{equation}
   \psi(G) \psi(H)  = \psi(G \cup H) \psi(G \cap H)   
   \label{determinantformula}
\end{equation}
holds and $\psi(K_d) = \phi(G) = (-1)^{f(G)} = -1$ for $d>1$ and $\psi(K_1)=1$. 
\end{coro}
\begin{proof}
The first statement follows directly from the theorem and uses the fact that the explicit
knowledge of $f(G)$ outs it is a {\bf valuation}, an integer-valued functional $f$ 
on the set of graphs or simplicial complexes
satisfying $f(G \cup H) = f(G) + f(H) - f(G \cap H)$. 
The second statement follows from the fact that the number of odd-dimensional simplices
in $K_d$ is odd if $d>1$. The reason is that the $f$-vector of a complete graph $K_{d+1}$ is 
explicitly given in terms of Binomial coefficients $(B(d+1,1), \dots, B(d+1,d+1))$ so that
$f(G)=2^{d}-1$ for $d>0$. In the case $d=0$, we have $f(G)=0$, otherwise $f(G)$ is odd. 
\end{proof} 

Let $b(G)$ denote the number of even-dimensional complete subgraphs of $G$. 
The {\bf join} of two graphs $G=(V,E),H=(W,F)$ is defined as the graph 
$(V \cup W,E \cup F \cup \{ (v,w) \; v \in V,w \in W \; \})$. The join operation has the same
properties as in the continuum: the join of two discrete spheres for example is again a
discrete sphere. The join of a graph $G$ with a $0$-sphere $S_0 = (V,E)=( \{a,b\}, \emptyset)$
is the {\bf suspension} of $G$. The octahedron from example is the suspension of the cyclic
graph $C_4$ and repeating the suspension construction on the octahedron and beyond 
produces all cross polytopes. An other consequences of the unimodularity theorem is:

\begin{coro}
If $G$ is the topological join of $K_1$ with a graph $H$, then 
$$   \psi(G) = \psi(H) (-1)^{b(H)} = (-1)^{\chi(H)} \; . $$
Especially, if $G$ is the suspension of $H$, then $\psi(G)=\psi(H)$. 
\end{coro}
\begin{proof}
Each newly added odd-dimensional simplex in $G$ corresponds to an
even-dimensional simplex in $H$. To see the second part, note that
the suspension is obtained by performing 
an second join operation over $H$ so that we get again $\psi(H)$. 
\end{proof} 

{\bf Remarks.}  \\
{\bf 1)} The unimodularity theorem is clear for disjoint graphs $F,G$ or if $F$ is a subgraph of $G$. 
A special case is if $F,G$ intersect in a single vertex. 
Then the Fredholm characteristics of $F$ and $G$ multiply. This special case is 
related to the known formula \cite{SimonTrace} (Corollary 8.7)
$$ \det(1+A+B) \leq \det(1+A) \det(1+B)  $$
because in the case of two subgraphs of a large complete host graph,
intersecting in a point, the adjacency matrix of $F \cup G$ is $A+B$. 
This is no more true for the connection graph:
if we join $F,G$ at a point, there are many simplices $x,y$
from different graphs which join, so that $A'+B'=(A+B)'$ is no more true. 
Still, the unimodularity theorem implies in that case that 
$$ \det(1+A'+B') = \det(1+A') \det(1+B')  $$
for the adjacency matrices of the two connection graphs $F',G'$ if $F \cap G = K_1$.
We see also that except for $K_2$, removing an edge from $K_d$ removes an even number of 
odd-dimensional complete subgraphs. The reason is that $2^k$ is even for positive $k$ and odd for $k=0$. 
For example, when removing an edge from $K_4$, we remove a tetrahedron and an 
edge to get a kite graph. The $f$-vector encoding the cardinalities of the
complete subgraphs changes from $(4,6,4,1)$ to $(4,5,2,0)$.  \\

{\bf 2)} The unimodularity theorem implies for example that 
the path integral sum can be replaced with one single permutation: 
for example, we can enumerate the odd dimensional simplices as $x_1,x_2, \dots,x_k$, then
define a map $x_i \to y_i$ from odd to even dimensional simplices by just dropping the first coordinate. 
Define $S_i(x_i)=y_i,S_i(y_i)=x_i$ and $S_i(x)=x$ for any other $x \neq x_i,x \neq y_i$. 
The transformation $S(x)=S_1(x) \cdots S_k(x)$ has the signature $\psi(G)$. While $S$ is generated by     
transpositions, it is not a transposition itself in general. It is not a transposition for a kite graph
for example. 
But it can be, like in the case $C_4$. While some permutations are continuous, the just mentioned one is not. 
The closed set $\{x\}$ of a vertex $x$ is mapped into an open set $\{ (xy) \}$. 
The map is not continuous.
While on $G=K_2$, there are $3!$ different transformations, only two of them are continuous 
and only the identity is a continuous flow. The other two transpositions $x \to (xy)$ and
$y \to (xy)$ are two signature $-1$ transformations but they are not continuous. 
The sum $1+(-1)+(-1)=-1$ is $\psi(G)$.  \\

The next corollary expresses the Fredholm characteristic of a ball by the Euler
characteristic of its boundary: 

\begin{coro}[Fredholm characteristic of unit ball]
For any graph $G=(V,E)$ and every vertex $x \in V$, we have
$\psi(B(x)) = (-1)^{\chi(S(x))}$.
\end{coro}
\begin{proof}
The number of odd-dimensional simplices in $B(x)$ not in $S(x)$
is equal to the number of even-dimensional simplices in $S(x)$.
Therefore, $f(B(x)) = f(S(x)) + b(S(x)) = \chi(S(x)) + 2 b(S(x))$.
Now exponentiate
$$  (-1)^{f(B(x))} = (-1)^{\chi(S(x))}  (-1)^{2 b(S(x))} = (-1)^{\chi(S(x))} \; . $$
(We could write this as $-(-1)^{i(x)}$, where $i(x) = 1-\chi(S(x))$ is a
Poincar\'e-Hopf index at $x$.)
\end{proof}

{\bf Examples:} \\
{\bf 1)} If $G$ is a graph for which every unit sphere is
a discrete sphere of Euler characteristic $0$ or $2$, 
then the unit ball $B(x)$ has $\psi(B(x)) =1$. 
For an icosahedron $G$ for example, where every unit ball is a wheel 
graph $W_5$ with $5$ spikes, there are $10$
edges in each unit ball.  \\

{\bf 2)} While the just mentioned corollary shows that for 
discrete spheres $S(x)$ with $\chi(S(x)) \in \{0,2\}$,
the unit balls always have $\psi(G)=1$, there
are spheres of arbitrary large dimension which can come both with
$\psi(G)=-1$ or $\psi(G)=1$. Here are examples:
Start with one dimensional sphere $C_n$ with odd $n$. It has $\psi(G)=-1$.
Every suspension is again a sphere and $\psi$ does not change under the suspension
operation. 

\section{Refining the simplicial complex}

Besides enlarging the base set $V$ of a graph $G=(V,E)$, there is an other possibility to
deform the geometry of a graph: we can {\bf refine the simplicial complex structure} imposed on it.
The simplest simplicial complex on a finite set is the {\bf 0}-skeleton, where the set
of subsets of $V$ is the set $\{ \{x\} \; | \; x \in V \; \}$. Since there are no paths
except the trivial path with signature $1$, the Fredholm characteristic of such a $0$-
dimensional space is $1$ and the unimodularity result is obvious in that case as no 
odd dimenional simplices exist then. \\

The next possibility is to take the {\bf 1}-skeleton complex $V \cup E$. This renders the
graph one-dimensional; it is the structure which is often associated with a graph, when
defining a graph as a one-dimensional simplicial complex. Also in this case, one can
see the unimodularity theorem. But it is less obvious already. First of all, we have
$f(G) = |E|$ as no higher dimensional simplices are present.
In the case when the graph is a tree, then there are no closed paths beside $K_2$ paths.
In order to analyze this, we have to see that for each edge there are two loops exchanging
the edge with the attached vertices. What happens if we merge two edges is that the two
transpositions merge to two circular loop depending on the order and that there is a new
transposition exchanging edges. The functional $\psi$ is multiplicative. \\

\begin{figure}[!htpb]
\scalebox{0.4}{\includegraphics{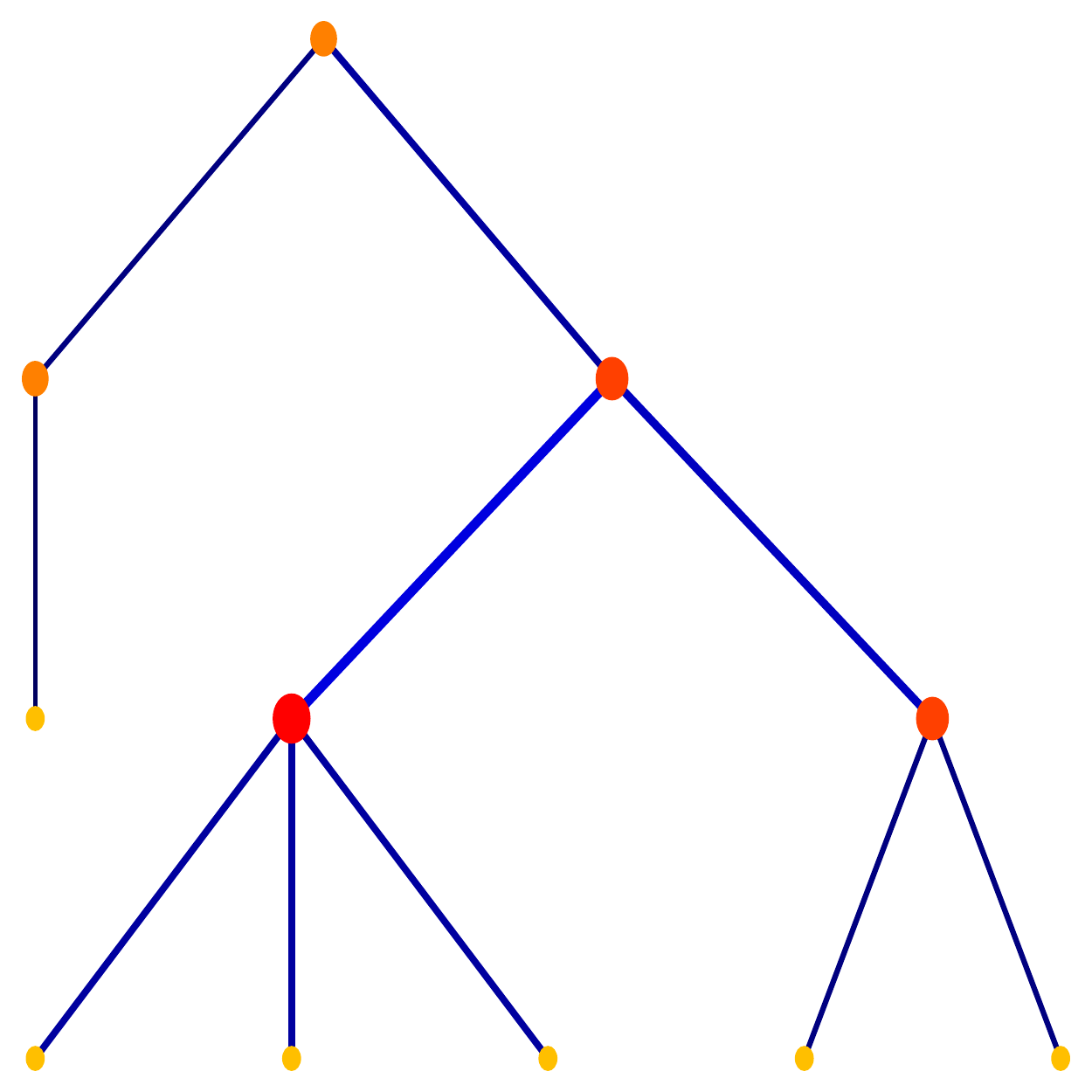}}
\scalebox{0.4}{\includegraphics{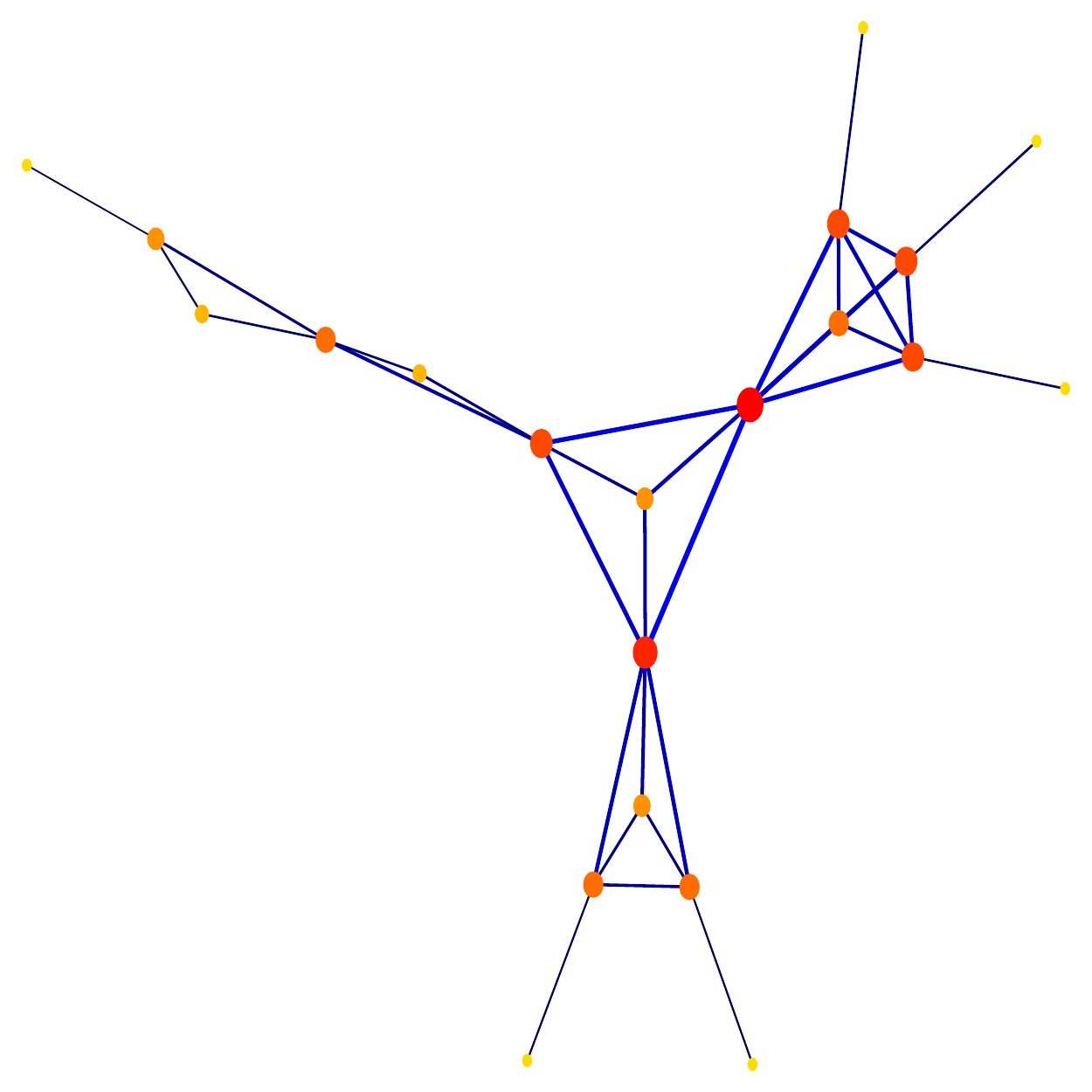}}
\caption{
A tree and its connection graph.
\label{figure1}
}
\end{figure}

If we start including two-dimensional simplices by ``switching on" the triangle,
the number of paths increases considerably. Lets call $C_3$ the triangular graph $K_3$
equipped with the one-dimensional simplicial complex. We have $\psi(C_3)=-1$ as there
are three odd dimensional simplices present, the 3 edges. We also have $\psi(K_3)=-1$
as there are no new odd-dimensional simplices.
The reason why the functional does not change when moving from $C_3$ to $K_3$ 
is because $\chi(S(x))=0$. 
But in the three dimensional case, because $\chi(S(x))=2$, the functional $\psi$ changes by
$-2$, becoming $\psi=-1$. Then again, the Euler characteristic of $S(x)=0$ and there is no change 
in $\chi(G)$.  \\

The unimodularity theorem holds more generally when the graph is equipped with a 
{\bf finite CW complex structure}. If a finite simple graph $G=(V,E)$ is equipped with such a 
simplicial or CW structure, it still defines a {\bf connection graph} 
and so a Fredholm determinant. 
The definition of a CW structure requires for a good notion of a ``sphere" 
in graph theory. Traditionally, a CW complex is a Hausdorff space
equipped with a collection of {\bf structure maps} from $k$-balls to $X$. 
As we have seen, a {\bf finite CW complex structure} works the same way, but 
cells don't need to be simplices any more. This is of practical value as we can work in 
general with smaller complexes when dealing with homotopy invariant notions like cohomology. 
It allows us to see the deformation of the simplicial structure as the process of adding or 
removing cells to the CW complex. 

\section{Extension proposition}

The following key result will allow to see how the Fredholm determinant 
changes if we add a cell to a CW complex. 
The Poincar\'e-Hopf indices are not necessarily $\{-1,1\}$-valued any more in general, because
$\chi(G)$ can take values different from $0$ or $2$. The Fredholm determinant extension formula 
needs a slightly more general extension process: \\

Let $G$ be a finite simple graph and let $H$ be a subgraph $H$ of $G$. 
Define the {\bf pyramid connection graph} $\tilde{G} = G' \cup_H \{x\}$ as the 
pyramid extension over the subgraph of $G'$ generated by the simplices which intersect 
a vertex in $H$. This new graph $\tilde{G}$ is equipped not with the Whitney
complex of the pyramid extension but with the union of the complex of $G'$ together with all 
simplices $y \cup \{x\}$, where $y$ is a simplex in $H$. One can build up any simplicial complex
as such. Given for example the triangle graph $G=(V,E)$ equipped with the 1-skeleton complex $C=V \cup E
= \{ a,b,c,ab,ac,bc \}$. The pyramid connection complex $\tilde{G}$ has now as vertices the elements
in $C$ together with a new element $x$. Now $x$ is connected to any element in $C$ which intersects $H$
additionally, any two elements in $C$ are connected if they intersect. In this case, the graph $\tilde{G}$
agrees with the connection graph of $K_3$ equipped with the Whitney complex.  \\

\begin{figure}[!htpb]
\scalebox{0.6} {\includegraphics{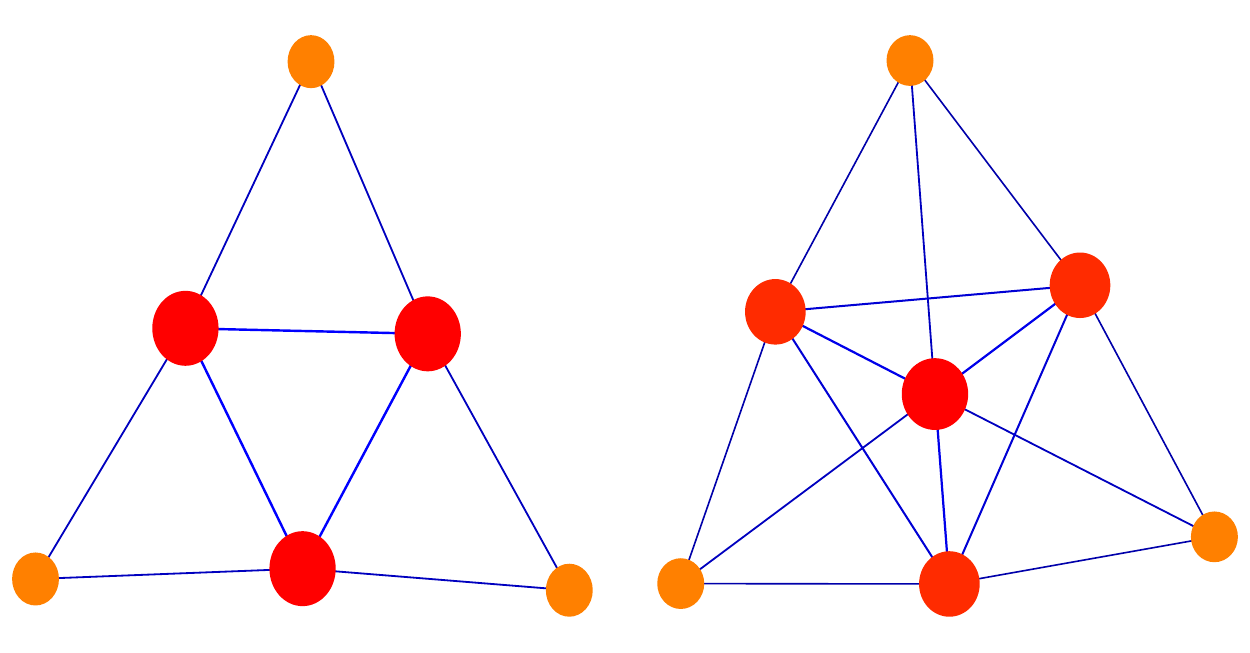}}
\caption{
The $1$-dimensional triangle or $C_3$ is the same graph as $K_3$ but it is 
equipped with the 1-skeleton complex 
(meaning that we disregard the 2-dimensional face which is present in  the Whitney complex). 
It is a sphere and has Euler characteristic $\chi(G) = 3-3$ as we disregard the 2-simplex. 
We see first its connection graph $G'$. After having added an other two dimensional cell 
to $H=G$, the connection graph $\tilde{G}$ is now graph isomorphic to the connection graph of $K_3$
equipped with the usual Whitney complex. The proposition tells that
$(-1) = {\rm det}(1+A(\tilde{G})) = (1-\chi(G)) {\rm det}(1+A(G')) = (1-0) (-1)$. 
\label{figure1}
}
\end{figure}

\begin{figure}[!htpb]
\scalebox{0.6} {\includegraphics{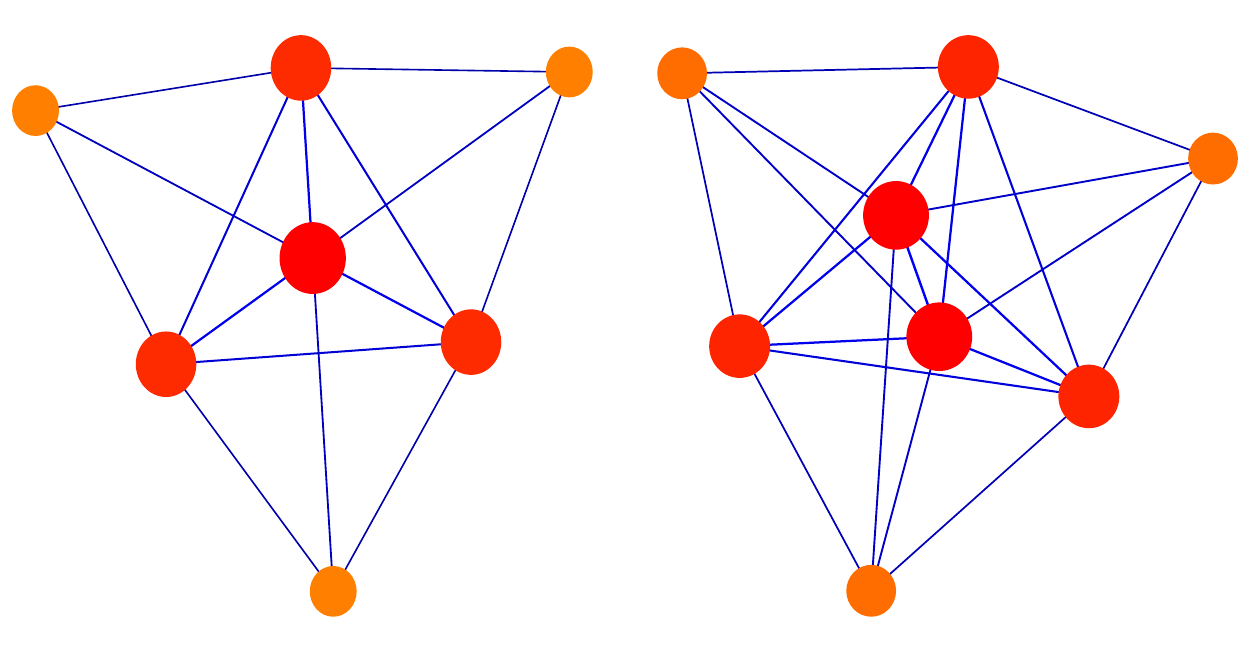}}
\caption{
As a comparison, we see $G=K_3$ equipped with the Whitney complex so that
$\chi(G)=1$. After adding an other vertex $x$ to $H=G$,
we get the graph $\tilde{G}$.
The proposition tells that
$0 = {\rm det}(1+A(\tilde{G})) = (1-\chi(G)) {\rm det}(1+A(G')) = (1-1) (-1)$. 
\label{figure1}
}
\end{figure}

Here is the result: 

\begin{propo}[Fredholm extension proposition] 
For any graph $G$ and any subgraph $H$ of $G$ we have
$$ \psi(\tilde{G}) = \psi(G' \cup_H \{x\}) = (1-\chi(H))  \psi(G') \; . $$
The same holds for simplicial complexes or CW complexes $G$ if $H$ is a sub complex.
\end{propo}
\begin{proof} 
{\bf (i)} The map $\eta: H \to \psi(G \cup_H x)-\psi(G)$ is an additive valuation. 
Proof. The functional super counts the number of 
new paths which appear by adding the cell $x$ and if a path passes through $x$,
then no other path can pass through it. For any two sub complexes 
$H,K$ we therefore have
$$  \eta(H \cup K) + \eta(H \cap H) = \eta(H) + \eta(K) \; . $$

{\bf (ii)} It follows that also
$$  X(H) \to 1-\psi(G \cup_H x)/\psi(G) $$ 
is a valuation: the value $\psi(G)$ is constant as we take $H$ as an
argument of the functional. \\

{\bf (iii)} The claim of the proposition holds if $H$ is a complete graph $K_{d+1}$ inside $G$.
We have to show  $\psi(\tilde{G}) = \psi(G' \cup_H \{x\})=0$. 
Proof. In the Fredholm matrix of $\tilde{G}$,
the row belonging to the $d$-simplex belonging to $H$ in $G$ and the row 
belonging to the new vertex $x$ all have entries $1$. The matrix is singular. \\

{\bf (iv)} Since by (iii), $\psi(X \cup_H x)=0$ if $H$ is a complete graph, 
the valuation $X$ takes the value $1$ for  every complete subgraph $H=K_k$.
Claim: a valuation with this property must be the Euler characteristic:
$$ X(H) = \chi(H)  \; . $$
Proof. By discrete Hadwiger, every valuation has the form $X(x)=a \cdot x$.
We know $X( (1) ) = 1, X( (2,1)  ) = 2 X( (1,0) ) + X(0,1)$ implying $X( (0,1) ) = -1$. 
Inductively this shows that the vector $a$ is $(1,-1,1,-1,1, \dots )$. An other, more
elaborate argument is to see that the assumption implies that the Euler characteristic of
a Barycentric refinement of a simplex is $1$ too, showing that $X$ is invariant under
Barycentric refinement. Since one knows that the eigenvalues and eigenvectors of the Barycentric
refinement operator and especially that Euler characteristic is the only fixed point of the
Barycentric operator, we are done. 
\end{proof} 

{\bf Examples.} \\
{\bf 1)} If $H=G=K_1$, then there is one empty path and one involution in the 
Leibnitz-Fredholm determinant, which is a path integral sum. The Fredholm characteristic is 
$1-1=0$.  \\
{\bf 2)} If $H=G=K_2$, then there are 2 graphs of length $3$, the empty graph and 
three involutions. Again we have $3-3=0$. \\
{\bf 3)} If $H=G=K_3$, then the Fredholm matrix 
$$ 1+A(\tilde{G}) = \left[
\begin{array}{cccccccc}
 1 & 1 & 1 & 1 & 1 & 1 & 1 & 1 \\
 1 & 1 & 0 & 0 & 1 & 1 & 0 & 1 \\
 1 & 0 & 1 & 0 & 1 & 0 & 1 & 1 \\
 1 & 0 & 0 & 1 & 0 & 1 & 1 & 1 \\
 1 & 1 & 1 & 0 & 1 & 1 & 1 & 1 \\
 1 & 1 & 0 & 1 & 1 & 1 & 1 & 1 \\
 1 & 0 & 1 & 1 & 1 & 1 & 1 & 1 \\
 1 & 1 & 1 & 1 & 1 & 1 & 1 & 1 \\
\end{array} \right]  \; . $$

The proposition makes the path counting problem additive as $\chi(H)$ is
a valuation. It shows that path integral of the set of new paths 
which appear when adding $x$. \\

{\bf Examples:} \\
{\bf 1)} Assume that $H$ has $k$ vertices and no edges. For every of the vertices $y$
we can look at the paths $xy$ of length $2$. Since we can not have simultaneously 
two paths hitting $x$, the number of additional paths is $k \psi(G)$. They are 
counted negatively as they have even length. \\
{\bf 2)} If $H$ is $K_2$ then we have three cycles of length $2$, three cycles of length $3$,
and one path of length $4$ paths passing through $x$. The last one cancels the empty path.
We get zero.  \\
{\bf 3)} If $H$ is a complete graph of $k$ vertices, then we have a unit ball in the
connection graph which is a complete graph. The extended graph has Fredholm 
characteristic $0$. \\
{\bf 4)} If $H$ is a cyclic graph of $k \geq 4$ vertices. Now the sum over new path 
contributions is zero: 

\begin{figure}[!htpb]
\scalebox{0.6}{\includegraphics{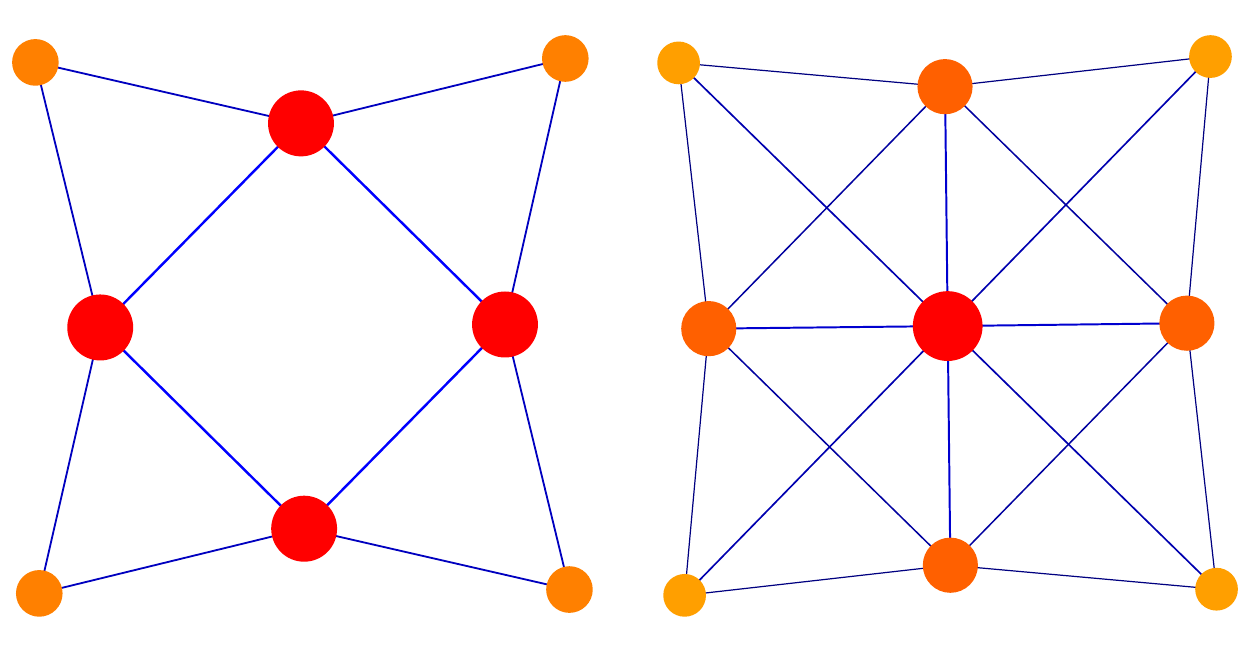}}
\caption{
We take $G=H=C_4$ with $\chi(G)=0$ and add a new cell $x$.
The Poincar\'e-Hopf index is $i(x)=1-\chi(H) = 1$.
The Fredholm characteristic of $G_x$ remains unchanged 
by the proposition. 
\label{figure1}
}
\end{figure}

\begin{figure}[!htpb]
\scalebox{1}{\includegraphics{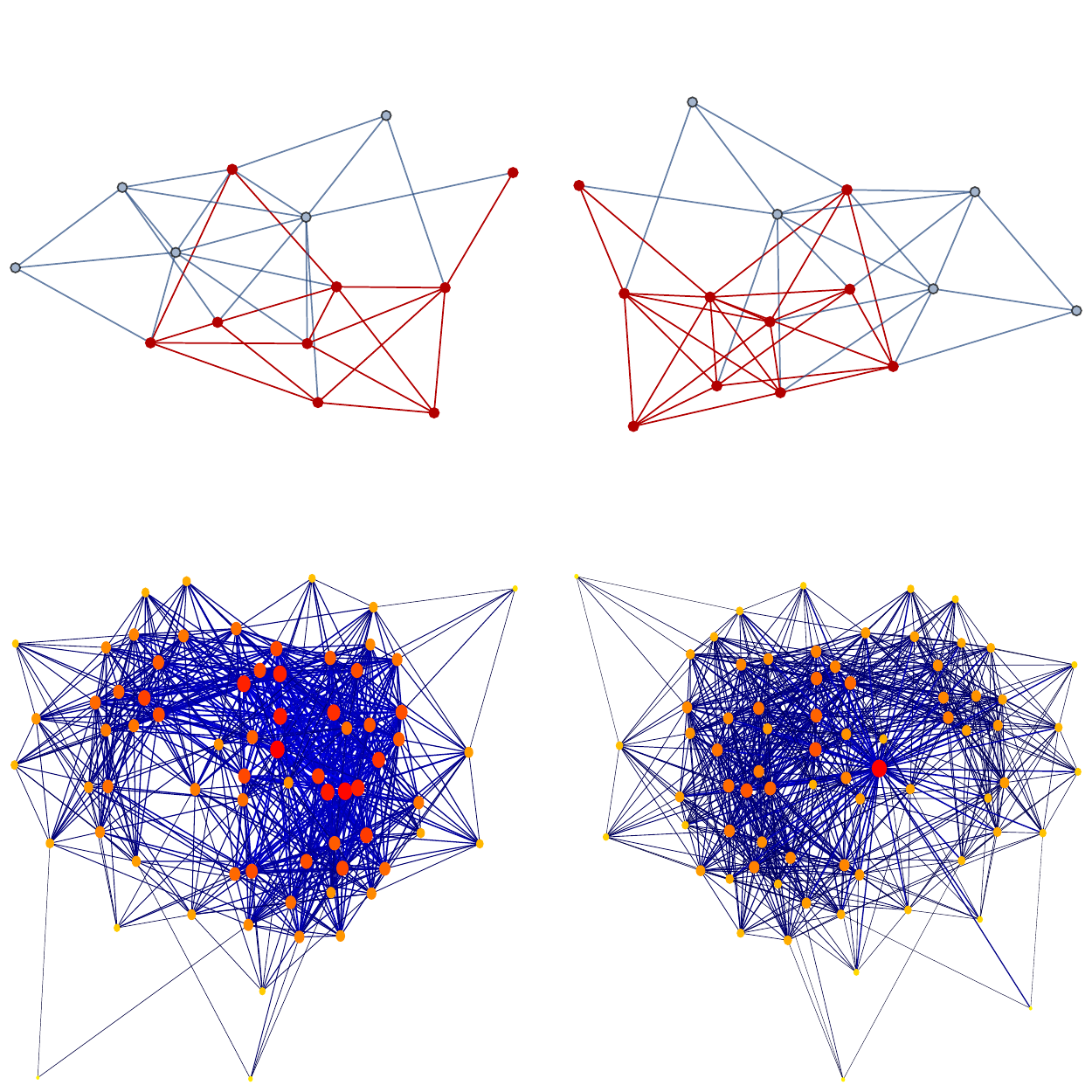}}
\caption{
We see a random graph G with $14$ vertices and Euler characteristic 
$\chi(G)=-4$. Inside is a random subgraph with 9 vertices of Euler 
characteristic $\chi(H)=-2$. The first picture shows H highlighted 
as a subgraph of G.  The second picture shows the pyramid extension 
$G_x$ of $G$ in which an additional vertex has been added. 
The Poincar\'e-Hopf index of the new index is $1-\chi(H) = 3$. 
The Fredholm characteristic of $G_x$ is $3 \cdot (-4) = -12$ by the proposition. 
\label{figure1}
}
\end{figure}

While $\psi(\tilde{G}) \in \{-1,1\}$ is not true any more in general, the 
proposition shows why the unimodularity theorem works. 
We can interpret the pyramid extensions over $k$-spheres as adding a 
$(k+1)$-dimensional cell or simplex. The proposition allows to relate Fredholm 
characteristic to Euler characteristics of extensions. In the case of
simplicial complex extensions, the $(1-\chi(H))$ are always $1$ or $-1$. \\

{\bf Examples:} \\
{\bf 1)} Let $G$ be a one point graph. The graph $\tilde{G}$ is $K_2$ and $\psi(K_2)=(1-\chi(G)) \psi(K_1) = 0$. \\
{\bf 2)} Let $G$ be the two point graph without edges. Then $\tilde{G}$ is the line graph $L_3$ and
$\psi(L_3) = (1-\chi(G)) \psi(G)=-1$. Indeed $L_3=K_2'$. \\
{\bf 3)} If $H$ is contractible, then $\psi(G' \cup_H \{x\})=0$. For example, if $H$ is a one point graph $\{y\}$
we just assign a new cell for which every simplex connected to $y$ is belonging. We have now a complete
subgraph $B(x)$. We know already that if a connection graph has a unit ball which is a complete graph,
then the Fredholm characteristic is zero. \\ 

\begin{coro}[Adding odd or even dimensional cells]
If $G$ is an even dimensional $k$-sphere, then 
$\psi(G' \cup \{x\}) = - \psi(G')$ and
if $G$ is an odd dimensional $k$-sphere, then 
$\psi(G' \cup \{x\}) = \psi(G)$. 
\end{coro}

This implies that if we
add an odd dimensional simplex to a simplicial complex, the
Fredholm characteristic changes sign and if we add an 
even dimensional simplex, then the Fredholm characteristic stays the same. 

\section{Prime and prime connection graphs}

For ever integer $n \geq 2$ let $V$ be the set of natural numbers in $\{2,3,\dots,n\}$ 
which are square free. Connect two integers if one is a factor of the other \cite{CountingAndCohomology}.
The corresponding graph $G_n$ is the {\bf prime graph}. One can see it as
the part $f \leq n$ of the Barycentric refinement of the complete graph on the
spectrum $P$ of the integers $\mathbb{Z}$, where $f(x)=x$ is the counting function.
The {\bf prime connection graph} $H_n$ has the same vertex set $V$ but now, two integers
are connected if they have a common factor larger than $1$ \cite{Experiments}. The two graphs allow us to 
illustrate some of the theorems. For the graph $G_n$, the Poincar\'e-Hopf theorem and the
Euler-Poincar\'e relation to cohomology is interesting as it allows to express the Mertens function 
interms of Betti numbers.  In the case of $H_n$, we can now formulate a consequence of 
the proposition which is the analogue of Poincar\'e-Hopf and the unimodularity result relating
the Fermi characteristic with a Fredholm determinant of the adjacency matrix. We see $G_n$
as part of the Barycentric refinement of the complete graph on the spectrum of the integers
$\Z$ and see $H_n$ as the connection graph of that complete graph. This point of view comes
into play when seeing {\bf square free integers} as simplices in a simplicial complex. \\

Despite the fact $H_n$ is not directly the connection graph of an other graph (it is part of
the connection graph of an finite graph, the complete graph with vectex set ${\rm spec}(\Z)$),
lets still denote by $\psi(H_n)$ the Fredholm characteristic, the Fredholm determinant 
of the adjacency matrix of $H_n$. 
Let $i_f(x)= 1-\chi(S^-_f(x))$ denote the Poincar\'e-Hopf index of the counting function $f$
at $x$. It is $-\mu(x)$, the {\bf M\"obius function}. We have seen that Poincar\'e-Hopf
implies that $\chi(G_n) = 1-M(n)$, where $M(n)$ is the {\bf Mertens function}.  \\

We have now directly from the unimodularity theorem for $CW$ complexes
a multiplicative Euler characteristic formula analogous to 
$$  \chi(G_n)   = \sum_{x \in V, x \leq n}  \mu(x) \; . $$

\begin{coro}
$\psi(G_n)   = \prod_{x \in V,x \leq n}  \mu(x)$.
\end{coro}
\begin{proof}
The left hand side is the Fredholm characteristic of the CW complex $G_n$. 
The right hand side is the Fermi characteristic of $G_n$. 
\end{proof}


Here is a comparison between the graph $G_n$ and $H_n$. Lets define 
$\pi(G) = \sum_k (-1)^k b_k(G)$ as the value of the Poincar\'e polynomial 
evaluated at $-1$, where $b_k(G)$ is the $k$'th Betti number. The cohomological
data of course correspond exactly to the corresponding notions in the continuum.
An Evako $d$-sphere for example has the Poincar\'e polynomial $1+x^d$ and a contractible
space has Poincar\'e polynomial $1$. 

\begin{center} 
\begin{tabular}{|l|l|} \hline
 Prime graph graph $G_n$                             &   Prime connection graph $H_n$ \\ \hline
 introduced in \cite{CountingAndCohomology}          &   introduced in \cite{Experiments}  \\ \hline
 relation is divisibility                            &   relation is nontrivial GCD  \\ \hline
 in Barycentric refinement of ${\rm spec}(\Z)$       &   in connection graph of ${\rm spec}(\Z)$  \\ \hline
 Euler characteristic $\chi(G) = \sum_x \omega(x)$   &   Fermi characteristic $\phi(G) = \prod_x \omega(x)$ \\ \hline
 Poincar\'e-Hopf $\chi(G \cup x) -\chi(G) = i(x)$    &   proposition $\chi(G \cup x) = i(x) \chi(G)$  \\ \hline
 Euler-Poincar\'e $\chi(G) = \pi(G)$                 &   Unimodularity  $\phi(G) = \psi(G)$  \\ \hline
\end{tabular}
\end{center}

{\bf Example.} The prime connection graph $H_{30}$ is a graph with 18 vertices
$\{ 2, 3, 5, 6, 7, 10, 11, 13, 14, 15, 17, 19, 21, 22, 23, 26, 29, 30 \}$
and $39$ edges. We can look at $H_n$ as a CW complex, where the individual nodes
represent the cells. The Fredholm matrix of $H_n$ is
$$ 1+A = \left[
                 \begin{array}{cccccccccccccccccc}
                  1 & 0 & 0 & 1 & 0 & 1 & 0 & 0 & 1 & 0 & 0 & 0 & 0 & 1 & 0 & 1 & 0 & 1 \\
                  0 & 1 & 0 & 1 & 0 & 0 & 0 & 0 & 0 & 1 & 0 & 0 & 1 & 0 & 0 & 0 & 0 & 1 \\
                  0 & 0 & 1 & 0 & 0 & 1 & 0 & 0 & 0 & 1 & 0 & 0 & 0 & 0 & 0 & 0 & 0 & 1 \\
                  1 & 1 & 0 & 1 & 0 & 1 & 0 & 0 & 1 & 1 & 0 & 0 & 1 & 1 & 0 & 1 & 0 & 1 \\
                  0 & 0 & 0 & 0 & 1 & 0 & 0 & 0 & 1 & 0 & 0 & 0 & 1 & 0 & 0 & 0 & 0 & 0 \\
                  1 & 0 & 1 & 1 & 0 & 1 & 0 & 0 & 1 & 1 & 0 & 0 & 0 & 1 & 0 & 1 & 0 & 1 \\
                  0 & 0 & 0 & 0 & 0 & 0 & 1 & 0 & 0 & 0 & 0 & 0 & 0 & 1 & 0 & 0 & 0 & 0 \\
                  0 & 0 & 0 & 0 & 0 & 0 & 0 & 1 & 0 & 0 & 0 & 0 & 0 & 0 & 0 & 1 & 0 & 0 \\
                  1 & 0 & 0 & 1 & 1 & 1 & 0 & 0 & 1 & 0 & 0 & 0 & 1 & 1 & 0 & 1 & 0 & 1 \\
                  0 & 1 & 1 & 1 & 0 & 1 & 0 & 0 & 0 & 1 & 0 & 0 & 1 & 0 & 0 & 0 & 0 & 1 \\
                  0 & 0 & 0 & 0 & 0 & 0 & 0 & 0 & 0 & 0 & 1 & 0 & 0 & 0 & 0 & 0 & 0 & 0 \\
                  0 & 0 & 0 & 0 & 0 & 0 & 0 & 0 & 0 & 0 & 0 & 1 & 0 & 0 & 0 & 0 & 0 & 0 \\
                  0 & 1 & 0 & 1 & 1 & 0 & 0 & 0 & 1 & 1 & 0 & 0 & 1 & 0 & 0 & 0 & 0 & 1 \\
                  1 & 0 & 0 & 1 & 0 & 1 & 1 & 0 & 1 & 0 & 0 & 0 & 0 & 1 & 0 & 1 & 0 & 1 \\
                  0 & 0 & 0 & 0 & 0 & 0 & 0 & 0 & 0 & 0 & 0 & 0 & 0 & 0 & 1 & 0 & 0 & 0 \\
                  1 & 0 & 0 & 1 & 0 & 1 & 0 & 1 & 1 & 0 & 0 & 0 & 0 & 1 & 0 & 1 & 0 & 1 \\
                  0 & 0 & 0 & 0 & 0 & 0 & 0 & 0 & 0 & 0 & 0 & 0 & 0 & 0 & 0 & 0 & 1 & 0 \\
                  1 & 1 & 1 & 1 & 0 & 1 & 0 & 0 & 1 & 1 & 0 & 0 & 1 & 1 & 0 & 1 & 0 & 1 \\
                 \end{array}
                 \right] \; . $$
Its determinant is $-1$. 

\begin{figure}[!htpb]
\scalebox{0.3}{\includegraphics{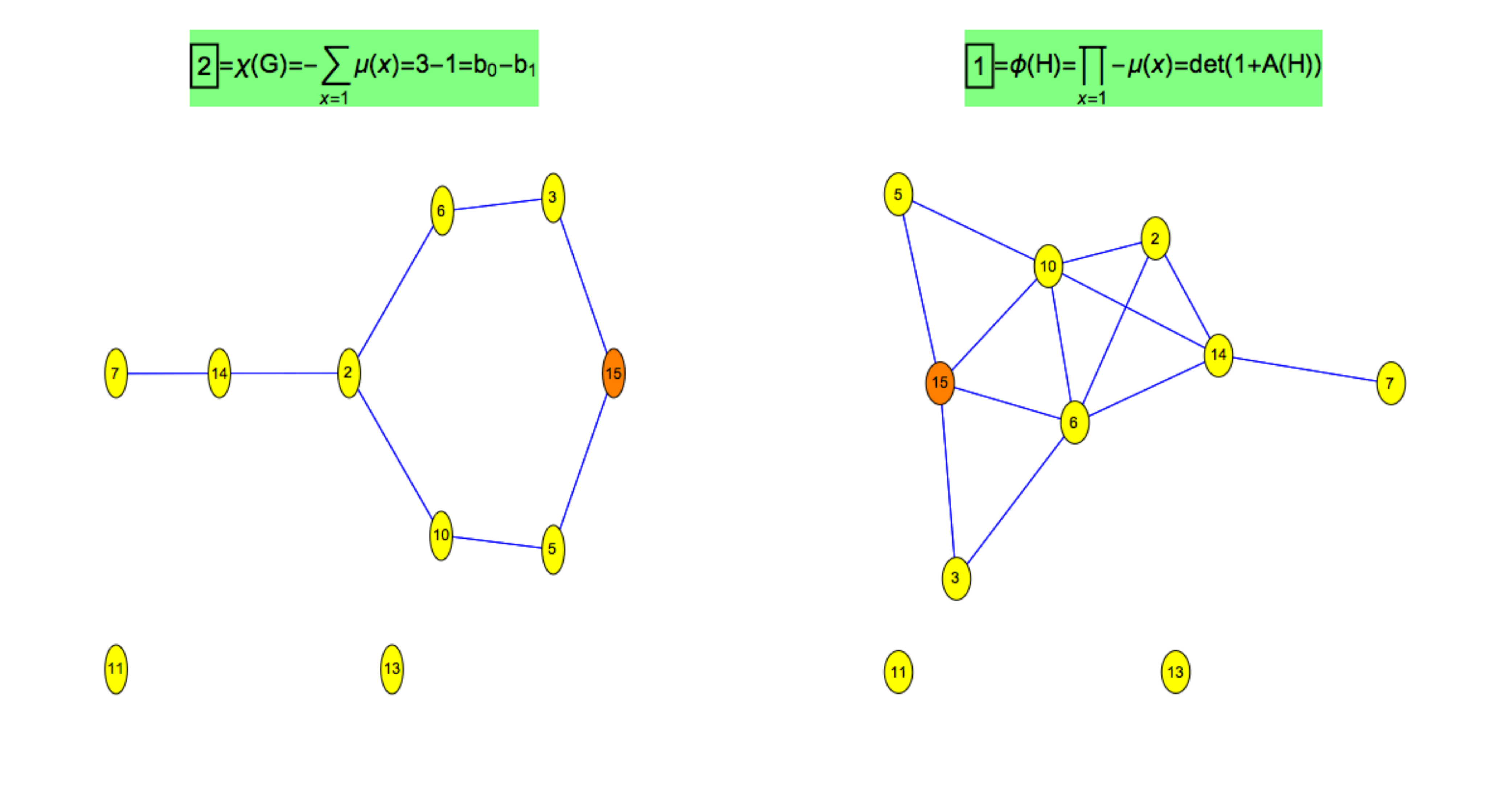}}
\caption{
The prime graph $G_{15}$ and prime connection graph $H_{15}$. 
\label{prime15}
}
\end{figure}

\begin{figure}[!htpb]
\scalebox{0.3}{\includegraphics{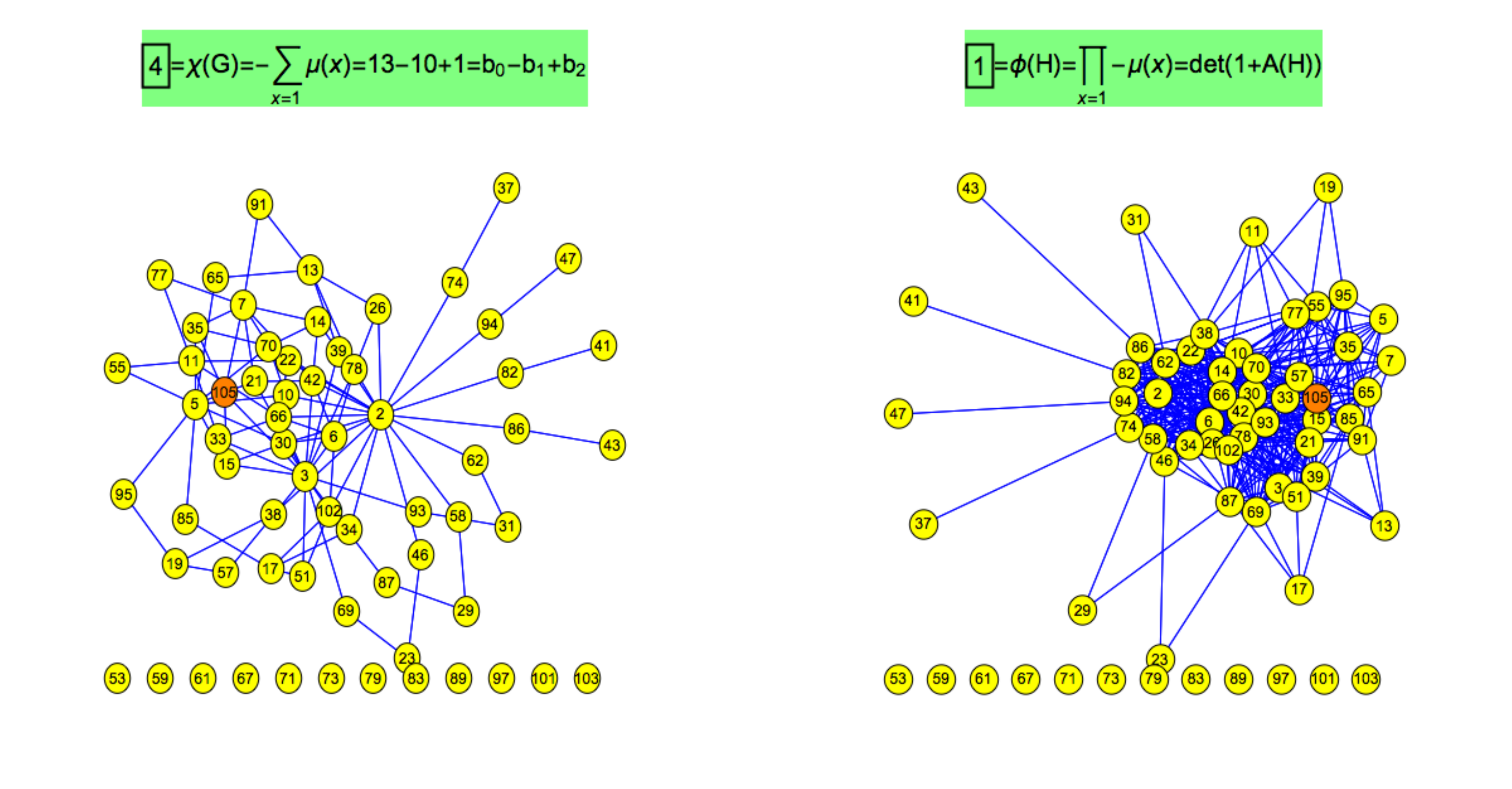}}
\caption{
The prime graph $G_{105}$ and prime connection graph $H_{105}$.
\label{prime15}
}
\end{figure}

\section{Remarks}

\paragraph{}
The unimodularity theorem does not hold for an arbitrary {\bf generalized matrix function} as defined by 
Marcus and Minc \cite{MarcusMinc}. It does not work for {\bf permanents} for example as we can 
check that $\psi$ is not a multiplicative valuation: 
the Fredholm permanent of the kite graph is $14$. The Fredholm permanent of the intersection is $2$.
The Fredholm permanent of the two triangular pieces is $6$. As mentioned in the introduction, the 
Fredholm permanent of a complete graph is $n!$, while the determinant is the number of derangements.
The fact that the product property or the unimodularity result can not hold for permanents is 
similar to the fact is that the Euler characteristic $\sum_k (-1)^k v_k$ 
is constant $1$ for a contractible graph and invariant under homotopy deformations, while the 
Bosonic analogue, the {\bf Kalai number} $\sum_k v_k$ counting the number of simplices in a graph
is a valuation but not invariant under homotopy deformations. \\

\paragraph{}
Why does the proof only work for the connection graphs and not for the graphs themselves?
Already for small graphs like $K_2$ or $K_3$, the Fredholm matrices are not unimodular any more. 
Indeed, for complete graphs, the Fredholm adjacency matrix is the matrix which has $1$ everywhere. 
It seems that what is needed is the high connectivity of the simplices in the connection graph:
paths in the unit ball in $G'$ of every original vertex $x$ are in 1-1 correspondence 
to permutations of the vertices in the ball $B(x)$ of $G'$. 
This is important as the sum of the signatures over all non-identity transformations
in $B(x)$ is $-1$. \\

\paragraph{}
We initially tried to prove the result by gluing two graphs $F,K$ to a larger graph $G=F \cup H$
and intersection $H=F \cap H$  and take
a permutation $x$ of $F'$ and $y$ of $K'$ and combine them to a permutation 
$x y$ in $G'$. The most natural choice is a composition of the permutations. While
$$ \psi(F) \psi(K) = (\sum_x {\rm sign}(x)) (\sum_y {\rm sign}(y)) = \sum_{x,y} {\rm sign}(x y) $$
holds, we are not able to match the composed transformations $x y$ with a transformation 
$z w$, where $z$ is a permutation in $G'$ and $w$ a permutation in $H'$. The product $x y$
is not necessarily a permutation any more in $G'$ and we need $w$ in $H'$
to achieve that. Now if $xy = x' y'$, then $x'^{-1} x = y y'^{-1}$ is a permutation in $H$.
In one of the simplest cases, where $F=K=L_3$ and $H=F \cap K=K_2$ and $F \cup K=L_4$. 
Now the number of permutations in $F'$ or $K'$ are 11. There are $121$ product transformations.
There are $39$ transformations in $G'$ and $3$ transformations in $H'$. 
There are only $117$ products. We see, we can not get all the pairs $xy$ with pairs $z w$. 
Cooking up a proof strategy which pairs the permutations up and then proves that that the signatures 
of the remaining sum up to zero has not yet worked.  \\

\paragraph{}
Let $G=(V,E)$ be a finite simple graph.
The determinant of the adjacency matrix $A$ is a super count of 
cyclic permutations of $V$. The
determinant of the Fredholm matrix $A+1$ is a super count of the
permutations of $V$. The Pseudo determinant of the Laplacian $L=B-A$
counts the number of rooted trees in $G$ and the determinant of
the Fredholm Laplacian $L+1$ counts the number of rooted forests in $G$. 
In each case we see that the determinant counts some zero or one-dimensional 
directed subgraphs.
What the significance of the unimodularity of the Fredholm matrix $1+A$ is
remains to be seen. It will depend on properties of the Green functions,
the entries of the matrix $g(A) = (1+A)^{-1}$
which is the Fr\'echet derivative of the map $A \to {\rm det}(1+A)$
from $n \times n$ matrices to the real axes \cite{SimonTrace} (Corollary 5.2). \\

\paragraph{}
It follows from the explicit formula for the $f$-vector 
that the Barycentric refinement $G_1$ of any graph $G$ is always a 
positive graph. 
There is an upper triangular matrix $A$ such that $\vec{v}(G_1) = A \vec{v}(G)$
for all simplicial complexes. The matrix is $A_{ij} = S(i,j) j!$, where
where $S(i,j)$ are the {\bf Stirling numbers} $S(i,j)$ of the second kind. Since $j!$
is even for $j>1$ all rows of $A$ beyond the first one are even. 
The Barycentric refinement matrix for simplicial complexes
with maximal dimension $4$ for example is 
$$ \left[
                 \begin{array}{ccccc}
                  1 & 1 & 1 & 1 & 1 \\
                  0 & 2 & 6 & 14 & 30 \\
                  0 & 0 & 6 & 36 & 150 \\
                  0 & 0 & 0 & 24 & 240 \\
                  0 & 0 & 0 & 0 & 120 \\
                 \end{array}
                 \right] \;. $$
Having seen that $\phi(G_1)=1$ for any graph $G$ or simplicial complex 
if $G_1$ is the Barycentric refinement, it follows 
by the unimodularity theorem that $\psi$ is {\bf not} a combinatorial invariant.  \\

\paragraph{}
Having discovered unimodularity experimentally in February 2016 and announced in October
\cite{MathTableOctober} it took us several attempts 
to do the proof presented here. Induction by looking algebraically at the matrices 
appeared difficult. An other impediment is the lack of an additive structure
on the category of graphs. It would be nice to see $G \to \log(\psi(G))$ 
as a group homomorphism for some larger group and then prove it for a basis only.
This has not yet worked: in the larger category of $Z_2$-chains we have also connection graphs
but $\psi$ does not extend. One can also look at the ring of graphs obtained by taking the 
connection graphs defined by a set of complete subgraphs (which form a Boolean ring).
But $\psi$ does not remain multiplicative on that larger structure.  \\

\paragraph{}
A finite simple graph $G$ with unimodular adjacency matrix $A$ is called an {\bf unimodular 
graph}. An example is the linear graph $L_{2k}$ for which the adjacency matrix has determinant 
$(-1)^k$. As this is not the same than having the Fredholm matrix $1+A$ unimodular, we 
here did not use the terminology ``unimodular graph".
Experiments show that the Fredholm matrices $B=1+A'$ of connection 
graphs often also have unimodular or zero determinant submatrices
obtained by deleting one row and one column which 
is equivalent that $B^{-1}$ is a $0,1,-1$ matrix.  In general, for a connection graph $G$,
the inverse $(1+A)^{-1}$ has only few elements different from $0,1,-1$. These appear
to be interesting {\bf divisors} to consider. \\

\paragraph{}
An {\bf abstract finite simplicial complex} $\X$ on a finite set $V$ 
is a set of subsets of $V$ such that if $A \in \X$ then also any subset of 
$A$ is in $\X$. Special simplicial complexes are {\bf matroids}, simplicial complexes
with the {\bf augmention property} telling that if ${\rm dim}(x)>{\rm dim}(y)$ 
then one can enlarge $x$ with and element in $y$ to get an element in the simplex. 
Our point of view was to see simplicial complexes on a finite set $V$ as a 
{\bf geometric structure} imposed on $V$ similarly as one imposes a 
{\bf set theoretical topology} $\O$ in topology or a {\bf $\sigma$-algebra} $\A$ in measure
theory. Like any topology, $\sigma$-algebra or simplicial complex $\X$ on $V$ 
defines a {\bf connection graph} $\X'$ of the structure: 
the vertices of $\X'$ are the sets in $\X$ and 
two such vertices are connected, if they intersect as subsets of $V$. 
The deformation of the simplicial structure gave more freedom than the deformation of the graph.
Besides complete graphs, one can define other simplicial complexes based on graphs. 
One can for example look at the matroid on the set of edges
in which the faces are the forests. This simplicial complex defines what one calls
a {\bf graphic matroid}. The Fermi characteristic of this matroid is the number of forests with
an odd number of edges. The connection graph of the matroid has as vertices the forests
and connects two if they intersect in at least one edge. The universality theorem
applies here too. For example, if $G=K_3$, there are 6 forests in $G$. The Fredholm matrix
of the connection graph 
$$ 1+A = \left[
                 \begin{array}{cccccc}
                  1 & 0 & 0 & 0 & 1 & 1 \\
                  0 & 1 & 0 & 1 & 0 & 1 \\
                  0 & 0 & 1 & 1 & 1 & 0 \\
                  0 & 1 & 1 & 1 & 1 & 1 \\
                  1 & 0 & 1 & 1 & 1 & 1 \\
                  1 & 1 & 0 & 1 & 1 & 1 \\
                 \end{array}
                 \right] $$
has determinant $-1$ so that the graphic matroid has Fredholm characteristic 
$\psi(G)=-1$. There are $3$ one-dimensional forests in $G$ so that also the 
Fermi characteristic is $-1$. 

\paragraph{}
A finite simple graph $G=(V,E)$ carries a natural associated simplicial complex, 
the {\bf Whitney complex} $\X$. It consists of all subsets of $V$ which are
complete subgraphs of $G$. 
The {\bf one skeleton complex} $\X=V \cup E$ is an other example of a simplicial
complex. Simplicial complexes do not form a Boolean ring: take the Boolean addition 
of $\X_2$ with $\X_1 \subset \X_2$, then this is the
complement of $\X_1$ in $\X_2$ which is no more a simplicial complex in general
as $\X_1=\{ \{1,2\},\{1\},\{2\}\}$ and $\X_2=\{ \{1\} \}$ shows.
The unimodularity theorem has no simple algebraic proof because the category of graphs
or the category of simplicial complexes only forms a {\bf Boolean lattice}
and not a {\bf Boolean ring}. An algebraic statement $\psi(G + H) = \psi(G) \psi(H)$
does therefore not work for the simple reason that the Boolean sum $G+H$ is not a graph.
Also, since the result does not extend to the {\bf ring of chains}, (the free Abelian group
generated by the simplices), we had to proceed differently. 
The Boolean algebra of the set of simplices in $G$ is a ring however so that 
for every element, we can define a connection graph. However, the functional $\psi$ 
does not extend as one can build any graph like that. 

\paragraph{}
The {\bf line graph} of a graph $G$ has the edges as vertices and two edges are
connected, if they intersect. It is a subgraph of the connection graph $G'$ of the
$1$-skeleton complex defined by the graph $G$. 
We mention this structure, as it is also sometimes called the {\bf intersection graph} 
of $G$. It plays a role in topological graph theory. If $d^*={\rm div}$ 
is the incidence matrix of $G$ which has as a kernel the cycle space, 
then $L_0 = d^* d=B-A$ is Kirchhoff Laplacian, where $B$ is the vertex degree diagonal matrix.
Now $|d d^*| + 2$ is the adjacency matrix of the line graph if $|A|_{ij}=A_{ij}$.
We see that the Fredholm matrix of $|d d^*|/2$ is an adjacency matrix of some graph.

\paragraph{}
The formula 
$$  \log(\det(1-z A)) = {\rm tr}(\log(1-z A)) = -\sum_k \frac{{\rm tr}(A^k)}{k} z^k $$
which converges for small $|z|$ gives an interpretation of the Fredholm determinant in 
terms of closed paths, which can self intersect. Note that $\det(1-zA)$ is a polynomial of degree $n$, where $n$
is the number of vertices in $G$. Since $\det(1-zA)^{-1} = \exp( \sum_k N_k z^k)$, where
$N_k$ is  the number of closed paths in the graph of length $k$. There is a relation with 
Ihera zeta functions, which satisfies $\zeta(z)^{-1} = {\rm det}(1-z A)$, where
$A$ is the {\bf Hashimoto edge adjacency matrix} of $G$. \\

\paragraph{}
Connection graphs could be defined for infinite, and even uncountable graphs. Its not clear 
yet which generalizations lead to Fredholm operators and so to Fredholm determinants.
Take the points of a circle $\mathbb{R}/\mathbb{Z}$ for example as the vertex set, fix two
irrational numbers $\alpha,\beta$ like $\alpha=\pi, \beta=e$ and $k \in \mathbb{N}$.
Connect two points if there exists $n,m \in \mathbb{Z}$ with $|n|,|m| \leq k$
such that $x-y-n \alpha - m \beta=0 {\rm mod 1}$. The points of the connection
graph is then the set of orbit pieces of the commutative $\mathbb{Z}^2$ action as
these are the ``simplices" in the original graph $G_{k,\alpha,\beta}$.

\paragraph{}
Connection graphs have a high local connectivity:
let $G$ be a connected graph which is not a complete graph and let
$G'$ is its connection graph. For any vertex $x \in V(G')$
the Fredholm determinant of $B(x) \subset G'$ is zero.
Proof. Assume first that $x$ is not the maximal central simplex $y$
containing $x$. Then, there are parallel columns $x,y$ in the
extended adjacency matrix.
Let $z$ be in $B(x)$. Then $z$ intersects $x$ and $y$.
The only way, the unit ball $B(x)$ in a connection graph
can have nonzero Fredholm determinant is if $G$ is a complete
graph and $x$ is the maximal central simplex.
For example, for $G=K_2=\{ \{a,b \; \} \; | \;  \{ (ab) \} \}$, the connection graph is
$L_3$. The unit sphere of $x=(ab)$ is $P_2$ which has Fredholm
determinant $1$. For all other $x$, the unit ball is $K_2$, which
has zero Fredholm determinant.

\begin{figure}[!htpb]
\scalebox{0.3}{\includegraphics{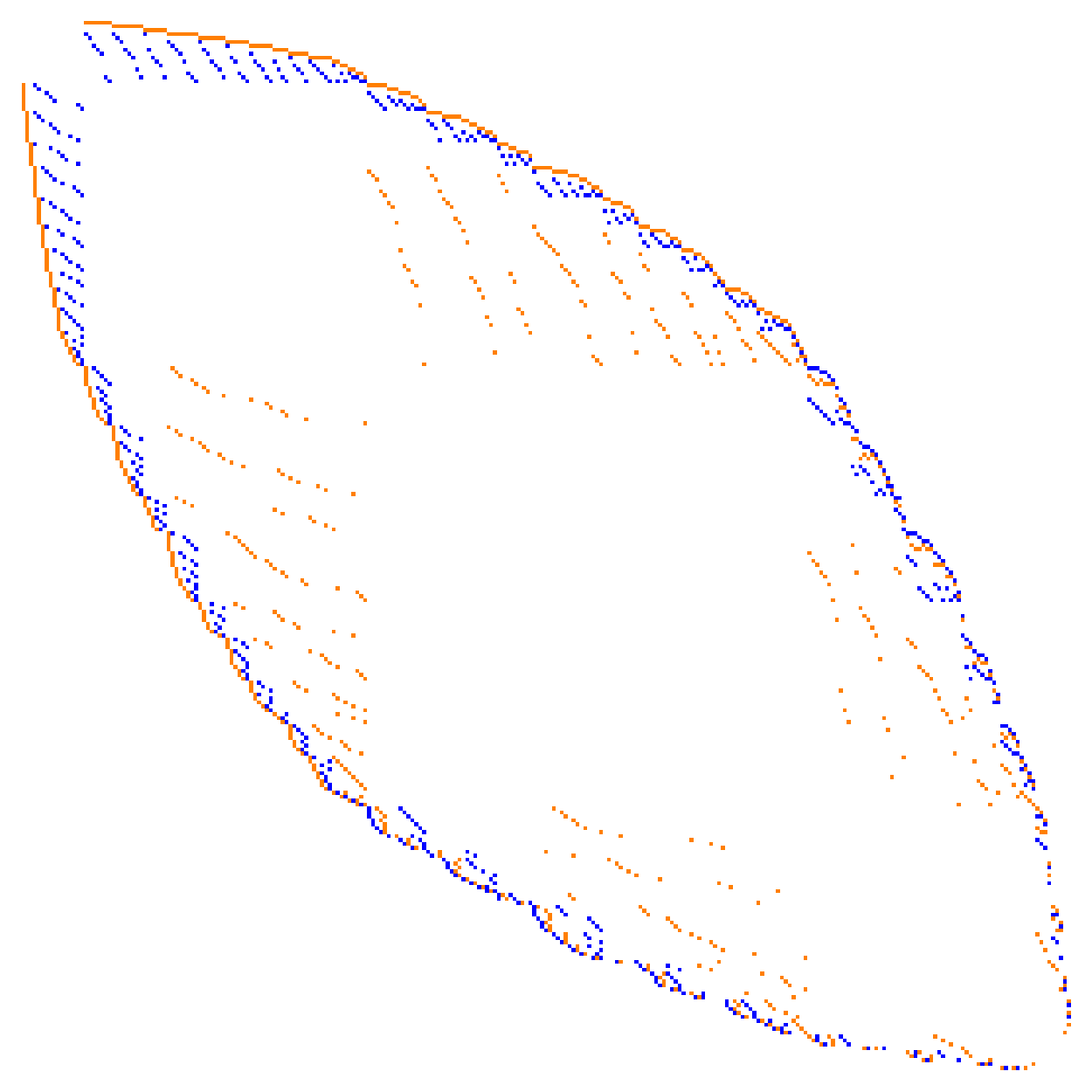}} 
\scalebox{0.3}{\includegraphics{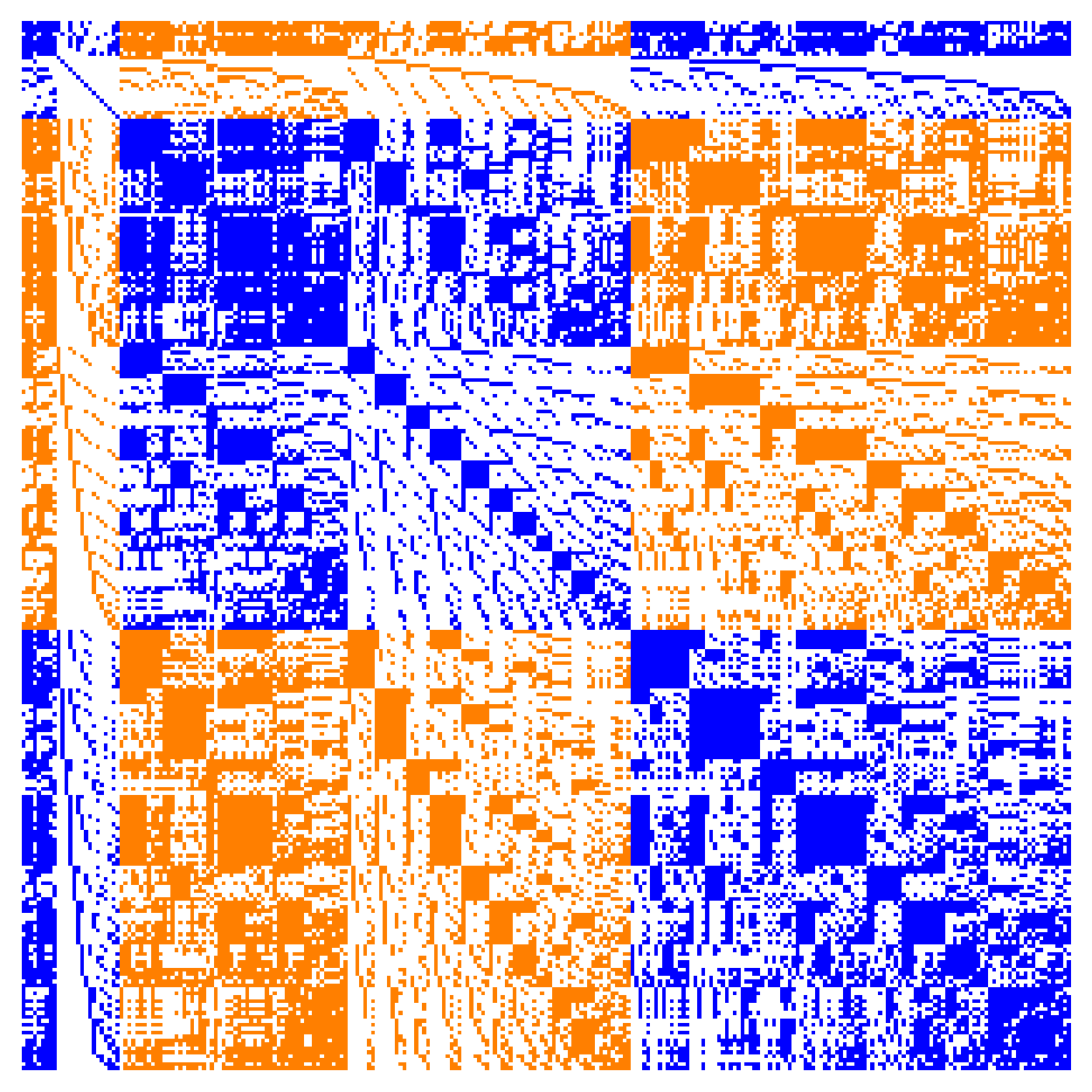}}
\scalebox{0.3}{\includegraphics{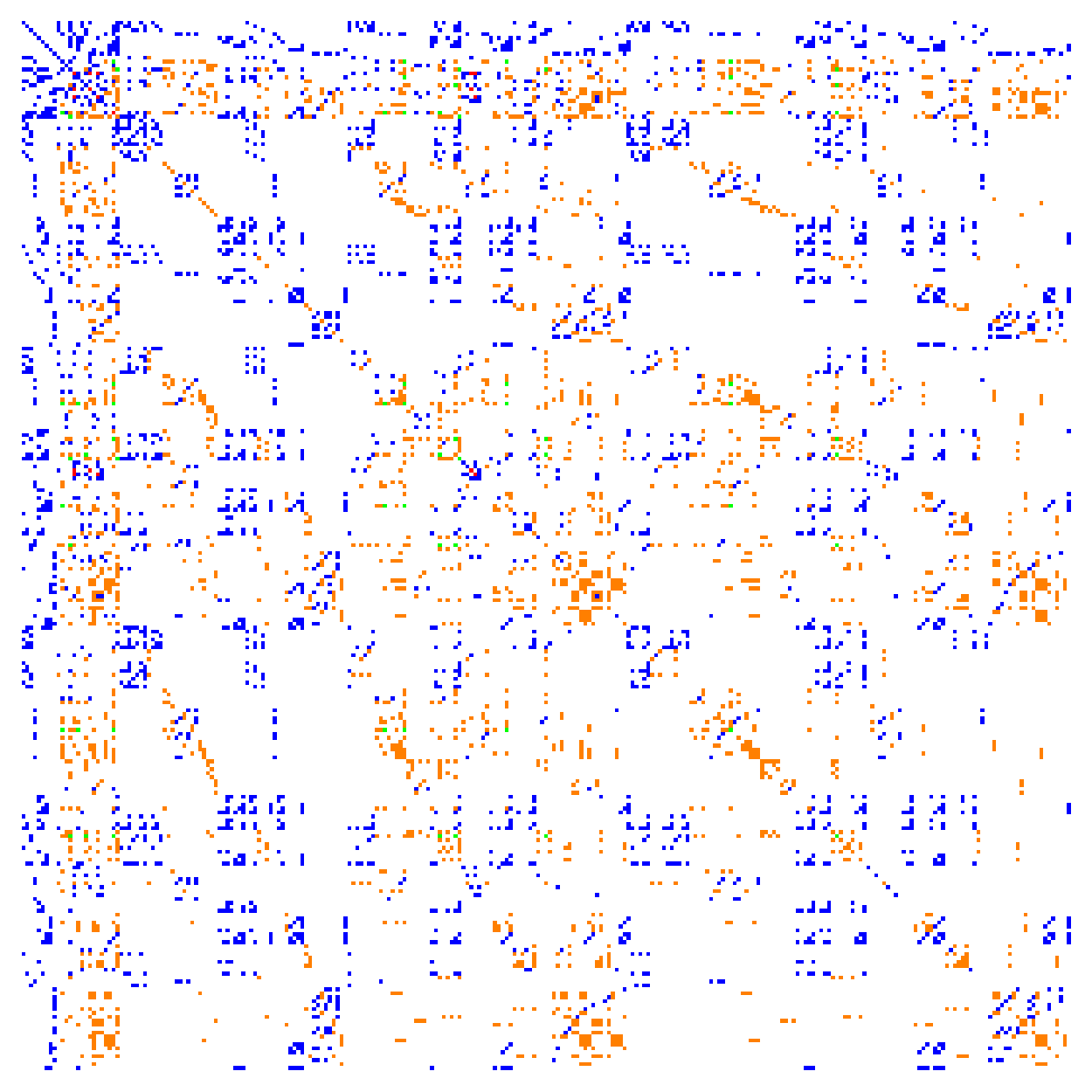}}
\caption{
We see the Dirac matrix $D$, Fredholm matrix and its inverse.
$D$ is up to a sign the adjacency matrix of the Barycentric refinement. 
Its entries $D_{xy}$ are nonzero whenever two different simplices $x,y$
have the property that one is contained in the other. The Fredholm matrix $B=(1+A')_{x,y}$
is nonzero if and only if $x,y$ have a non-empty intersection. Because of the $1$, the
intersection is counted also when $x=y$. The last picture shows the inverse of that
matrix $B$. Most of its entries are $0,1,-1$ but there are a few larger entries.
\label{figure1}
}
\end{figure}

\begin{figure}[!htpb]
\scalebox{0.4}{\includegraphics{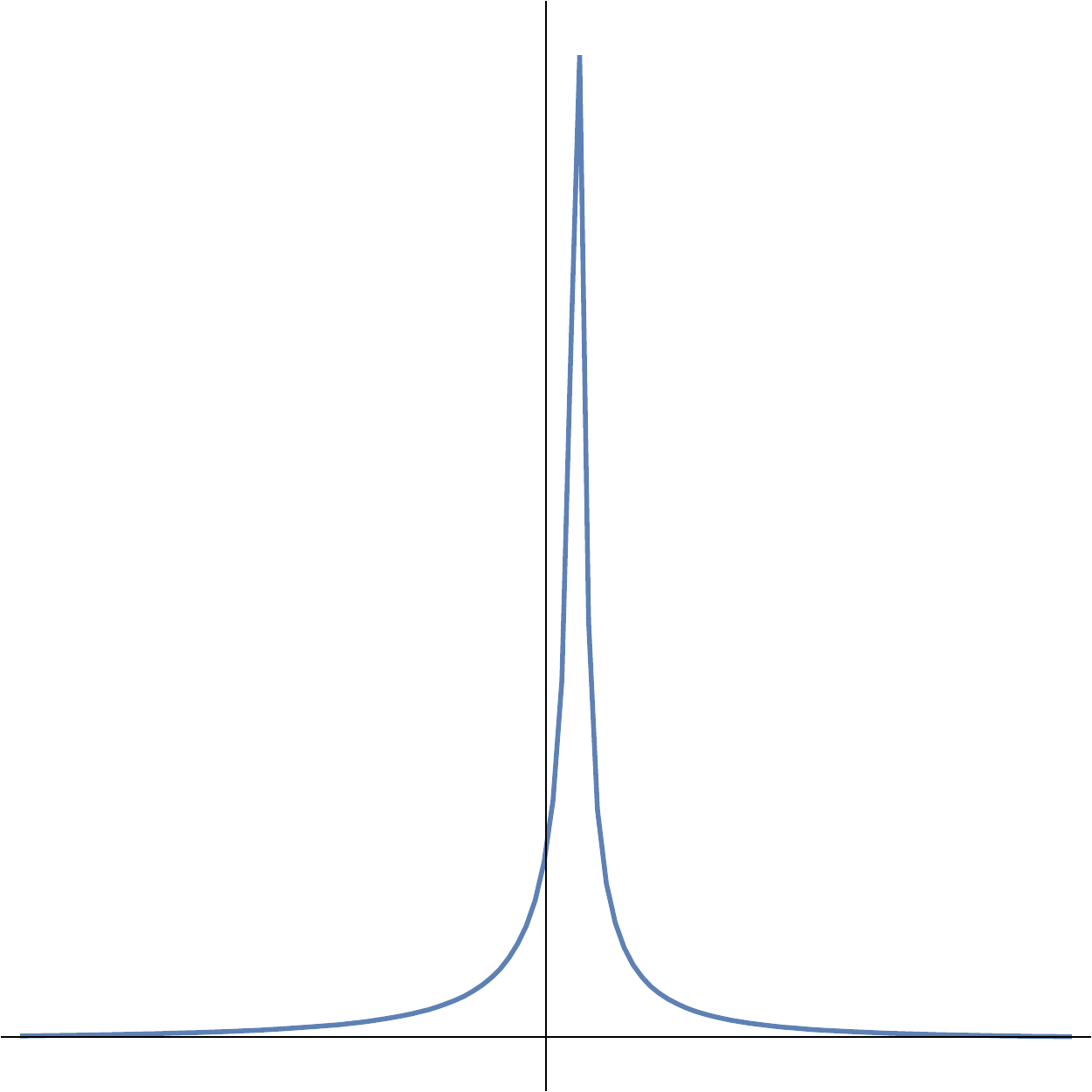}}
\scalebox{0.4}{\includegraphics{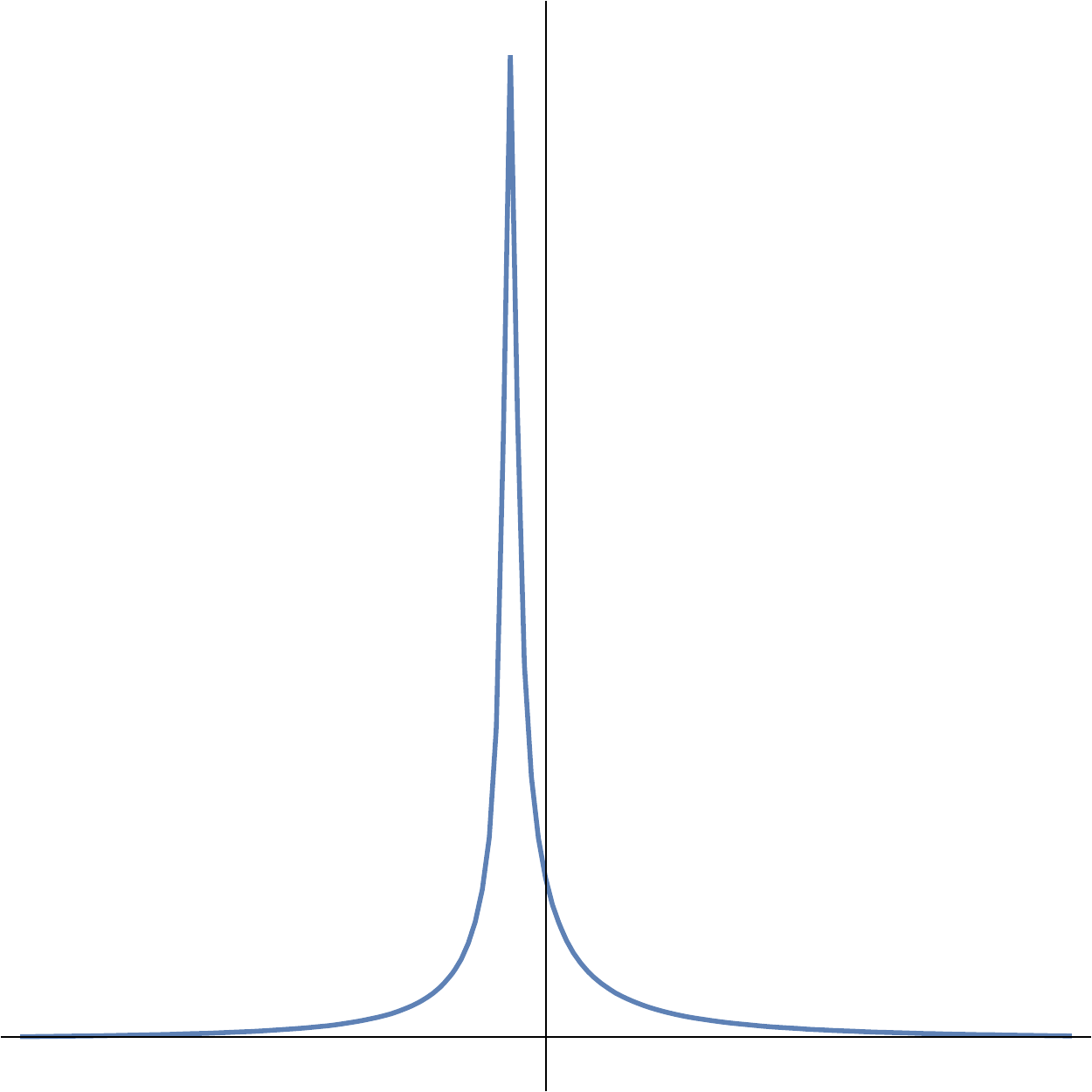}}
\caption{
In the first picture we see the distribution of the determinant 
the adjacency matrix of a random graph with $70$ nodes and edge probability $p=1/2$. 
We computed it for $10^7$ random graphs and plotted the distribution on $[-1,1]$ after
the random variable normalized so that the mean is zero and the standard deviation is $1$. 
The right picture shows the same distribution but for the Fredholm determinant. Again the 
mean is zero and standard deviation $1$ is the boundary of the interval. 
We see that both distributions appears to be singular
in the limit but that there is no quantization which happens for connection graphs,
where the value is either $1$ or $-1$. 
\label{distribution}
}
\end{figure}

Benjamin Landon and Ziliang Che, members of the Harvard random matrix group tell me that
current technology of matrix theory like \cite{TaoVu2011} could be close in being 
able to prove a central limit theorem for determinants of random graphs but that the 
existence of a central limit in that graph case still remains to be done. 

\section{Questions}

{\bf 1)} Let $c(p) = \limsup_{n \to \infty} P_{n,p}[ \psi=1 ]$, 
where $P_{n,p}$ is the counting probability measure on the 
Erd\H{o}s R\"enyi space $E(n,p)$. We believe the limsup is actually a limit for every $p \in [0,1]$ 
and that the limit is $1/2$ for every $p \in (0,1)$. We know $c(0)=1$ for $p=0$ because a graph $G$ without edges
has $\psi(G)=1$ and for $c(1)=0$, because a complete graph $G$ has $\psi(G)=-1$. 
Why do we think that the limit in general is $1/2$? Because of the central limit theorem for 
$Z_2$-valued random variables. What happens if we go from $E(n,p)$ to $E(n+1,p)$ is that we are 
given a distribution of a $Z_2$-valued random variable $(-1)^{b(H)}$ on subgraphs of $G$
get the new graph by adding a pyramid over $H$. Since we expect
the probability of picking a graph $H$ with $\psi(H)=1$ or $\psi(H)=-1$ is more and more $1/2$,
we essentially multiply with some independent $Z_2$-valued random variable when adding
a new vertex. By the central limit theorem
of random variables taking values in compact groups with a Haar measure, the limiting
distribution has to be the uniform distribution which means equal distribution on finite groups.
The case is not yet settled because we convolute each time with a distribution 
which we want to show to converge. This requires some estimates. Of course, a situation with
some nonzero thresholds with a probability switching from $0$ to $1/2$ would have been more
exciting, but it seems not to happen here; 
there is no reason in sight why the number of odd-dimensional subgraphs should pick a particular parity.  \\

{\bf 2)} It appears that both the function $G \to {\rm det}(A(G))$ 
as well as $G \to {\rm det}(1+A(G))$ have a continuous limiting 
distribution on random graphs when the functions are normalized so 
that the mean is $0$ and the standard deviation is $1$. Establishing this
would settle a central limit theorem for determinants. One usually looks
also at the random variable $\log|{\rm det}(1+A(G))|$ when looking at limits 
but that requires first establishing that the determinant $0$ has asymptotically 
zero probability. \\

{\bf 3)} What is the structure of the inverse matrices of $1+A(G')$, where
$A(G')$ is the connection graph of the connection graph $G'$ of $G$? 
Here are examples: for many graphs, including the 
complete graph, the wheel graphs, cross polytopes or cycle graphs,
the inverse of the matrix $1+A(G')$ only takes the
values $-1,0,1$. For {\bf star graphs} with $n$ rays, we see that
the values of the inverse of $1+A(G)$ is $-(n-1),-1,0,1$. 
For the utility graph, or the Petersen graph, the values are 
$-2,-1,0,1$. 
On the $13$ Archimedian solids, we see 6 for which the value takes
values in $\{-2,-1,0,1\}$ and $7$, where the value takes value in 
$\{-1,0,1\}$. On the 13 Catalan solids, the value cases are
$\{-1,0,1\}$, $\{-3,-2,-1,0,1\}$, $\{-4,-2,-1,-1,0,1\}$
and $\{-4,-3,-2,-1,0,1\}$ occur. \\

{\bf 4)} We measure that the {\bf Green function values} given by the 
inverse matrix of 
$$  g_{ij} = [(1+A(G'))^{-1}]_{ij} $$ 
is the same for $G$ and the Barycentric refinement $G_1$ of $G$. 
We also measure that maximal and minimal value among the diagonals, the
{\bf Green function values} are independent of Barycentric refinement.
The diagonal entries $G_{ii}$ are up to a sign the 
Fredholm determinants of the graph $G' \setminus x_i$, where entry $x_i$ has been removed. 
If the entry $G_{ii}$ is $1$ or $-1$, there is a chance that the modified graph is 
the connection graph of a CW complex. 
The Green function value $g_{ij}$ is the Fredholm determinant of the structure obtained by snapping the
$ij$ connection between cell $i$ and $j$. As the so lobotomized graph is no more a
connection graph in general, the $g_{ij}$ can be different.
We have no explanation yet for these measurements. This would actually indicate
that the minimal values taken for example are a {\bf combinatorial invariants} of $G$. 
and that one could assign the invariant to the limiting continuum object. 
Following a suggestion of Noam Elkies, we looked also whether an edge refinement
does not change the range of the Green functions $[(1+A(G'))^{-1}]_{ij}$.
Indeed also to be the case and a first thing which should be checked theoretically.
The invariance under edge refinements would imply topological 
invariance under Barycentric refinements for graphs without triangles. \\

{\bf 5)} Because unimodular matrices form a group, we can define for 
all $n$, the group of all Fredholm connection matrices of subgraphs of
a given graph $G$. Assume $n$ is the number of simplices in $G$. 
If $H$ is a subgraph of $G$, then the simplices are
part of the simplices of $G$. All matrices therefore can be made
the same size $n \times n$. The set of all these Fredholm connection matrices 
now forms a subgroup of the unimodular group in $GL(n,Z)$. Can we characterize 
these groups somehow? \\

{\bf 6)} We know that unit balls of connection graphs have Fredholm characteristic $0$ 
and Euler characteristic $1$.
The structure of the unit spheres is less clear. We measure so far
that unit spheres of connection graphs either have Euler characteristic $1$
or $2$. The later case is more rare but appears already for the octahedron 
graph $G$. Lets see: given a simplex $x$, then the unit sphere consists of all 
simplices which intersect $x$. Now, there are three possibilities: either 
$y$ is a subgraph. The set of subgraphs is are all connected with each other
and we can homotopically reduce them to a point. 
Then there are all graphs which contain $x$. Also all these are connected
with each other and we can homotopically reduce them to a point. 
Finally, there are all graphs which intersect. We have to show that they
are always homotopic to an even dimensional sphere. \\

{\bf 7)} We tried to correlate $\psi(G)$ with cohomological data, both for simplicial
cohomology as well as connection cohomology. Simlilarly as homotopy deformations 
do change $\psi$ but not cohomology, this also appears to happen for connection
cohomology even so the later is not a homotopy invariant: the cylinder and
the M\"obius strip are homotopic but have different connection cohomology. \\

{\bf 8)} What is the relation between the f-vector of $G$ and the f-vector of $G'$?
Unless for Barycentric refinement, the $f$-vector of $G$ does not determine
the $f$-vector of $G'$. The Euler characteristic is not an invariant as the 
octahedron graph $G$ for which the connection graph $G'$ is contractible. 

\section{Illustrations}

\begin{figure}[htpb]
\scalebox{0.35}{\includegraphics{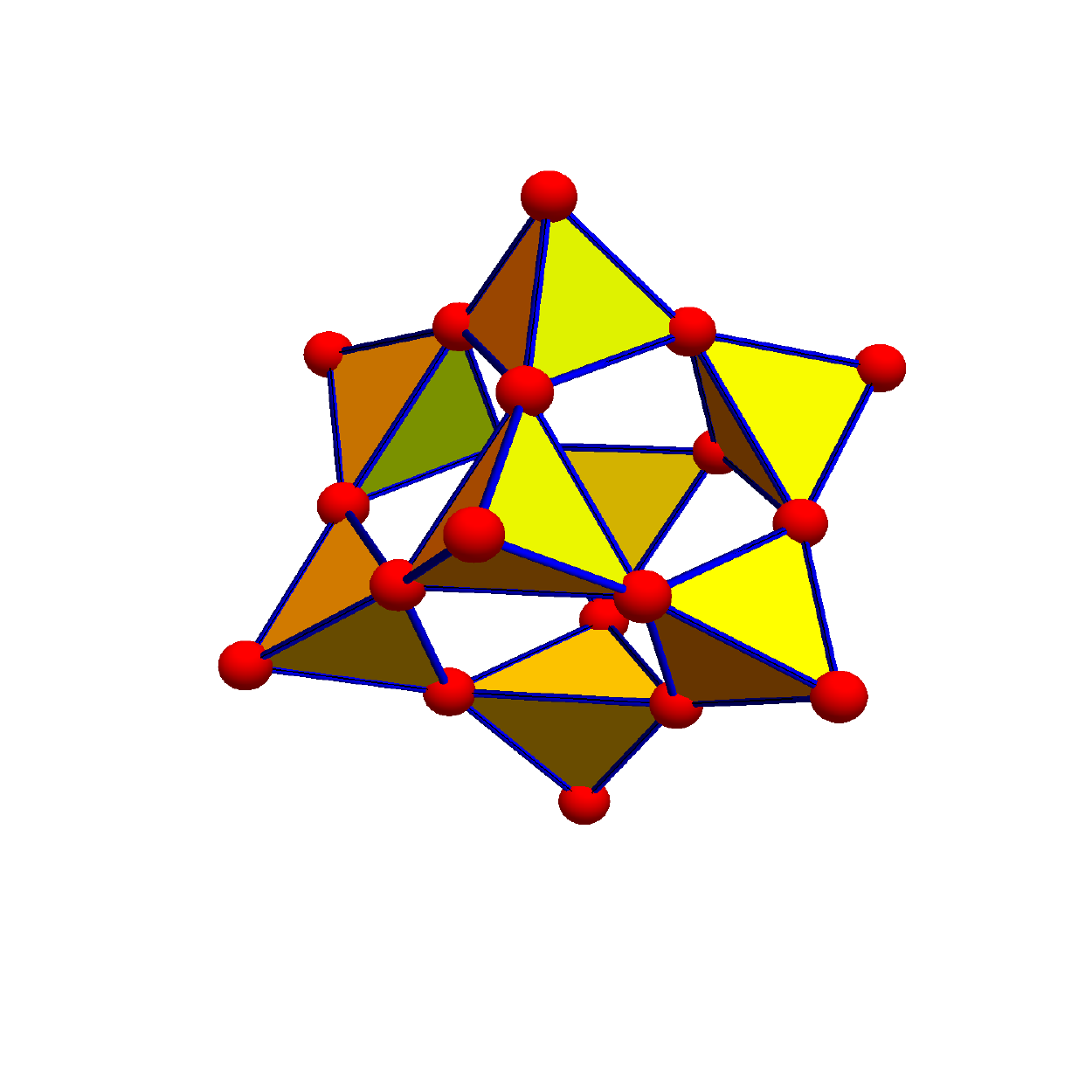}}
\scalebox{0.35}{\includegraphics{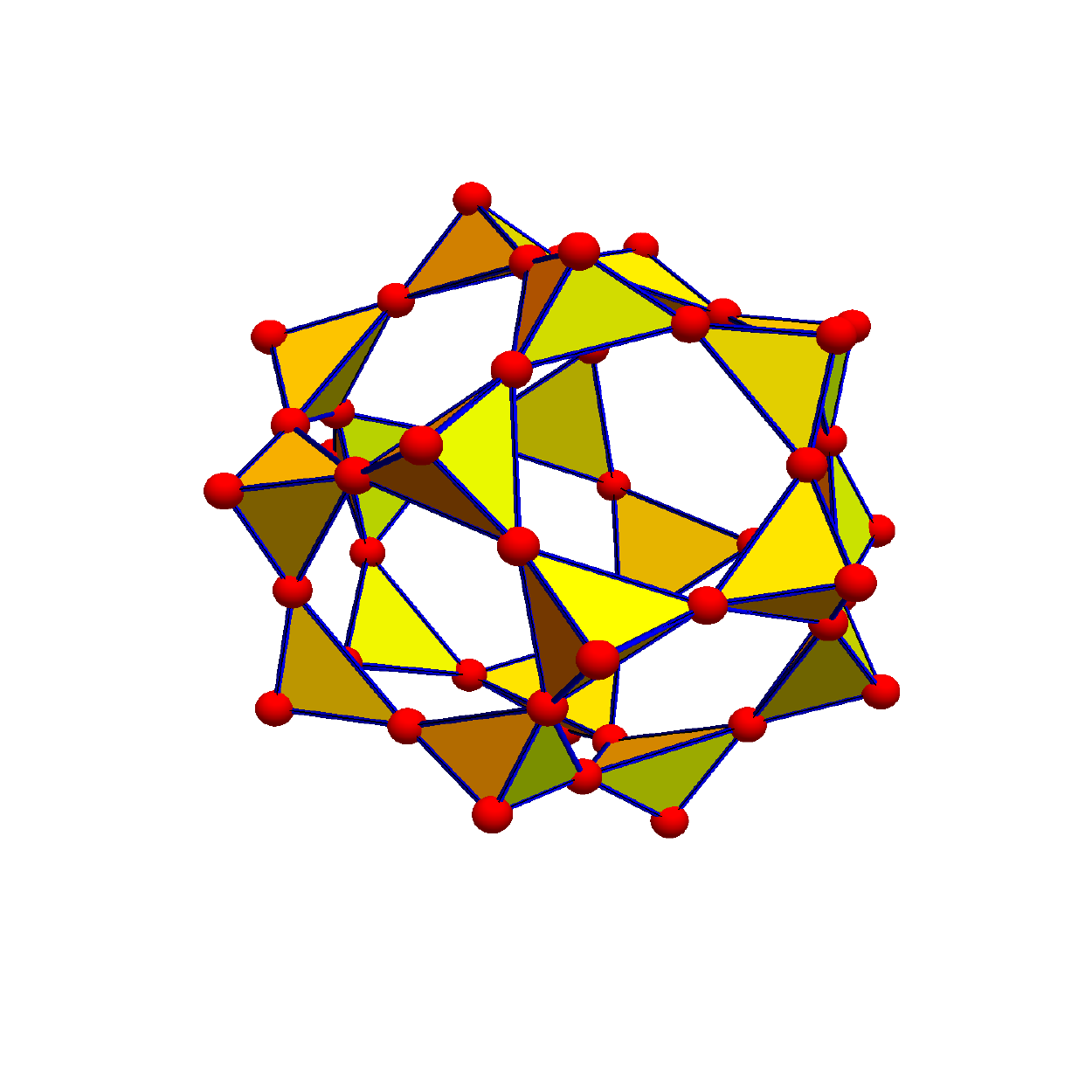}}
\caption{
The connection graph of the cube graph and the 
dodecahedron graph.
\label{figure7}
}
\end{figure}

\begin{figure}[htpb]
\scalebox{0.7}{\includegraphics{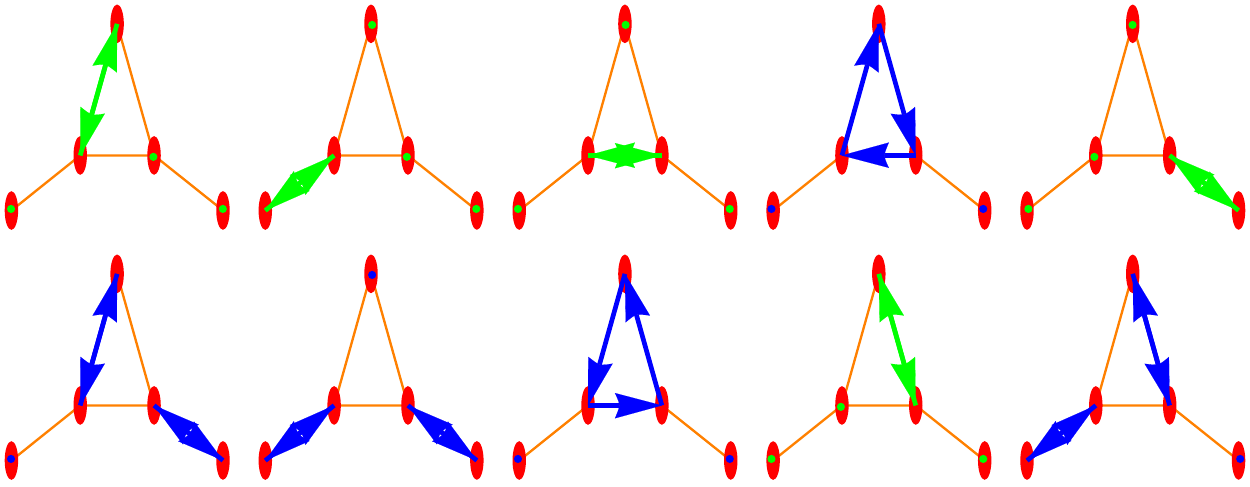}}
\caption{ 
There are $11$ one-dimensional oriented subgraphs of $G'$
if $G$ is the linear graph of length $2$. There are $5$
with odd signature and $6$ with even signature including
the empty path which is not shown. We have $\psi(G)=6-5=1$
and $\phi(G)=(-1)^2=1$. 
\label{figure7}
}
\end{figure}

\begin{figure}[htpb]
\scalebox{0.8}{\includegraphics{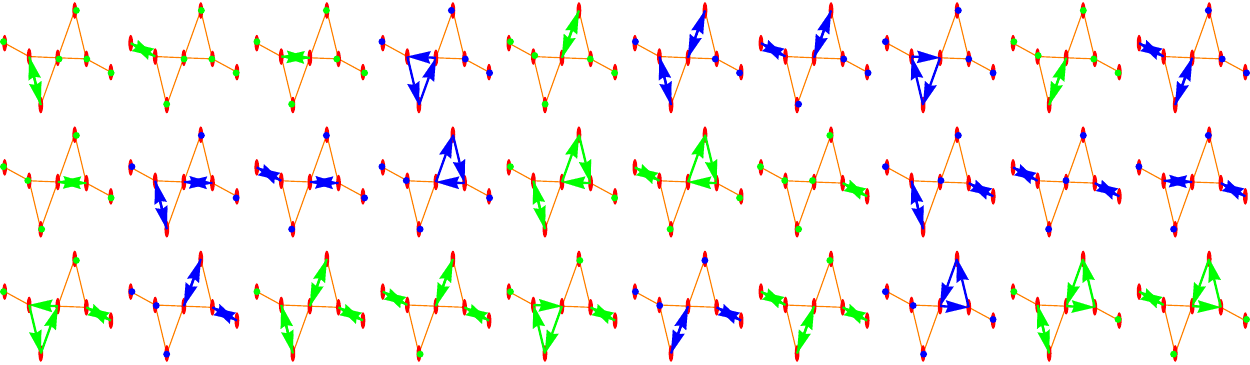}}
\caption{
There are $39$ one-dimensional oriented subgraphs of $G'$
if $G$ is the linear graph of length $3$. There are $20$
with odd signature and $19$ with even signature.
We have $\psi(G)=19-20=-1$ and $\phi(G)=(-1)^3=-1$. 
}
\end{figure}

\begin{figure}[htpb]
\scalebox{0.8}{\includegraphics{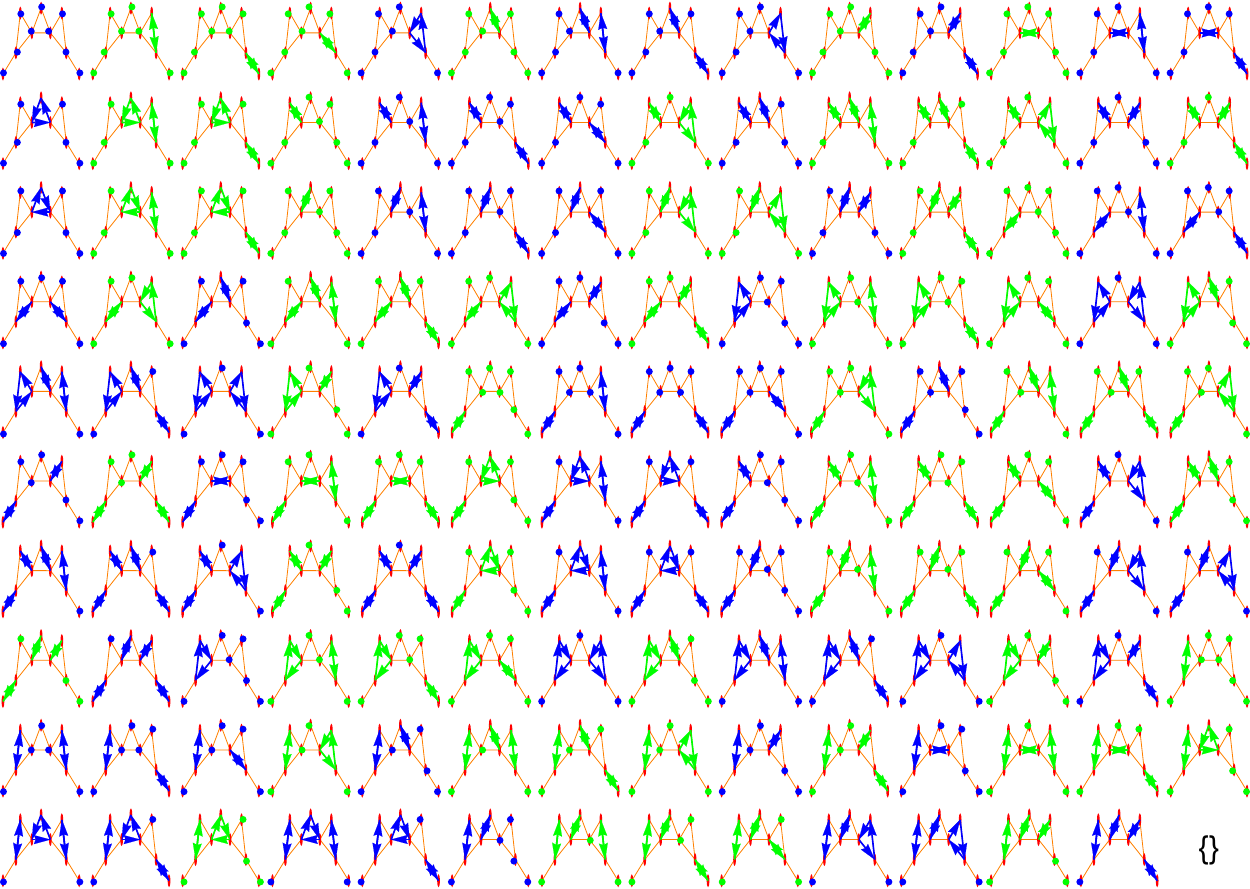}}
\caption{
There are $139$ one-dimensional oriented subgraphs of $G'$
if $G$ is the linear graph of length $4$. There are $70$
with even signature and $69$ with odd signature.
We have $\psi(G)=70-69=1$ and $\phi(G)=(-1)^4=1$.
}
\end{figure}

\begin{figure}[htpb]
\scalebox{1}{\includegraphics{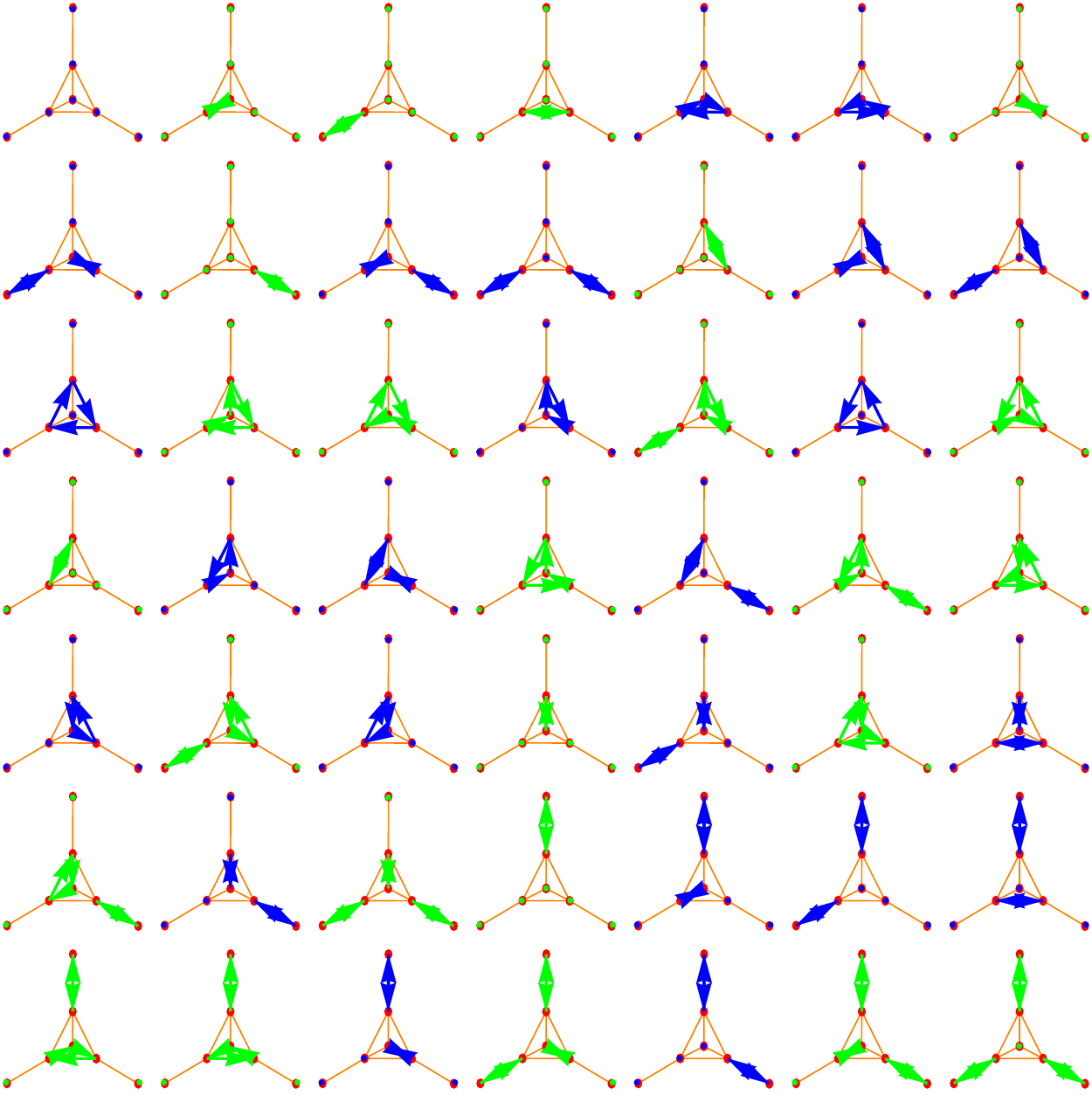}}
\caption{
There are 49 one-dimensional oriented subgraphs of $G'$
if $G$ is the star graph with three spikes. 25 of them
have odd signature. We have $\psi(G)=24-25=-1$ and 
$\phi(G)=(-1)^3=-1$. 
\label{figure7}
}
\end{figure}

\begin{figure}[htpb]
\scalebox{1.3}{\includegraphics{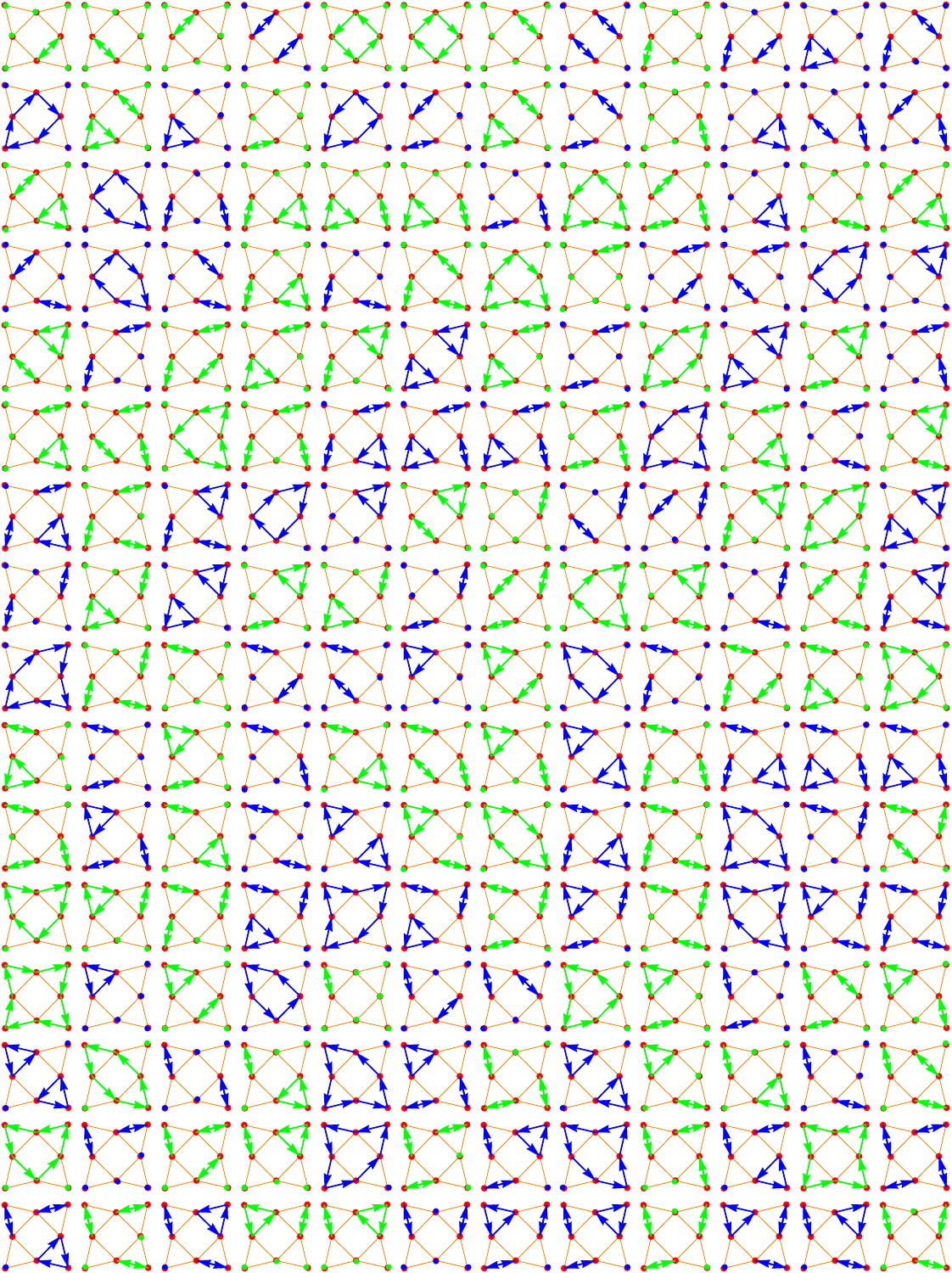}}
\caption{
There are $193$ one-dimensional oriented subgraphs of
$G'=C_4'$ for $G=C_4$.  The figure shows all except the empty path, 
$97$ of of them have signature $1$, and $96$ have signature $-1$. 
We have $\psi(G)=97-96=1$ and $\phi(G)=(-1)^4=1$. 
\label{figure4}
}
\end{figure}

\begin{figure}[htpb]
\scalebox{1.3}{\includegraphics{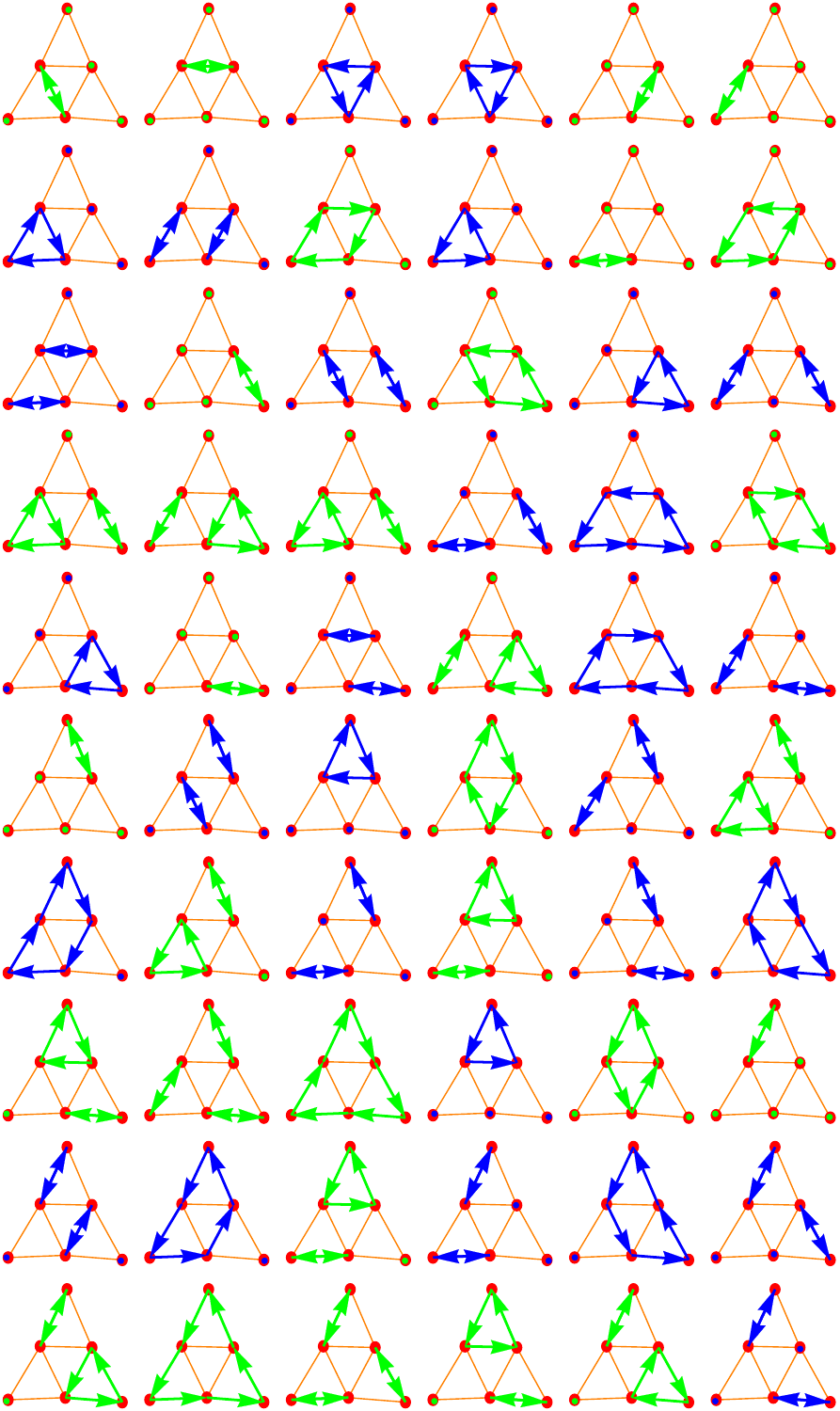}}
\caption{
There are $61$ one-dimensional oriented subgraphs of
$G=K_3$ equipped with the one dimensional skeleton 
complex $G'$. We see the $60$ paths which are different from the identity.
$30$ have signature $1$, and $31$ have signature $-1$. 
We have $\psi(G)=30-31=-1$ and $\phi(G)=(-1)^3=-1$. 
\label{figure5}
}
\end{figure}

\begin{figure}[htpb]
\scalebox{1.24}{\includegraphics{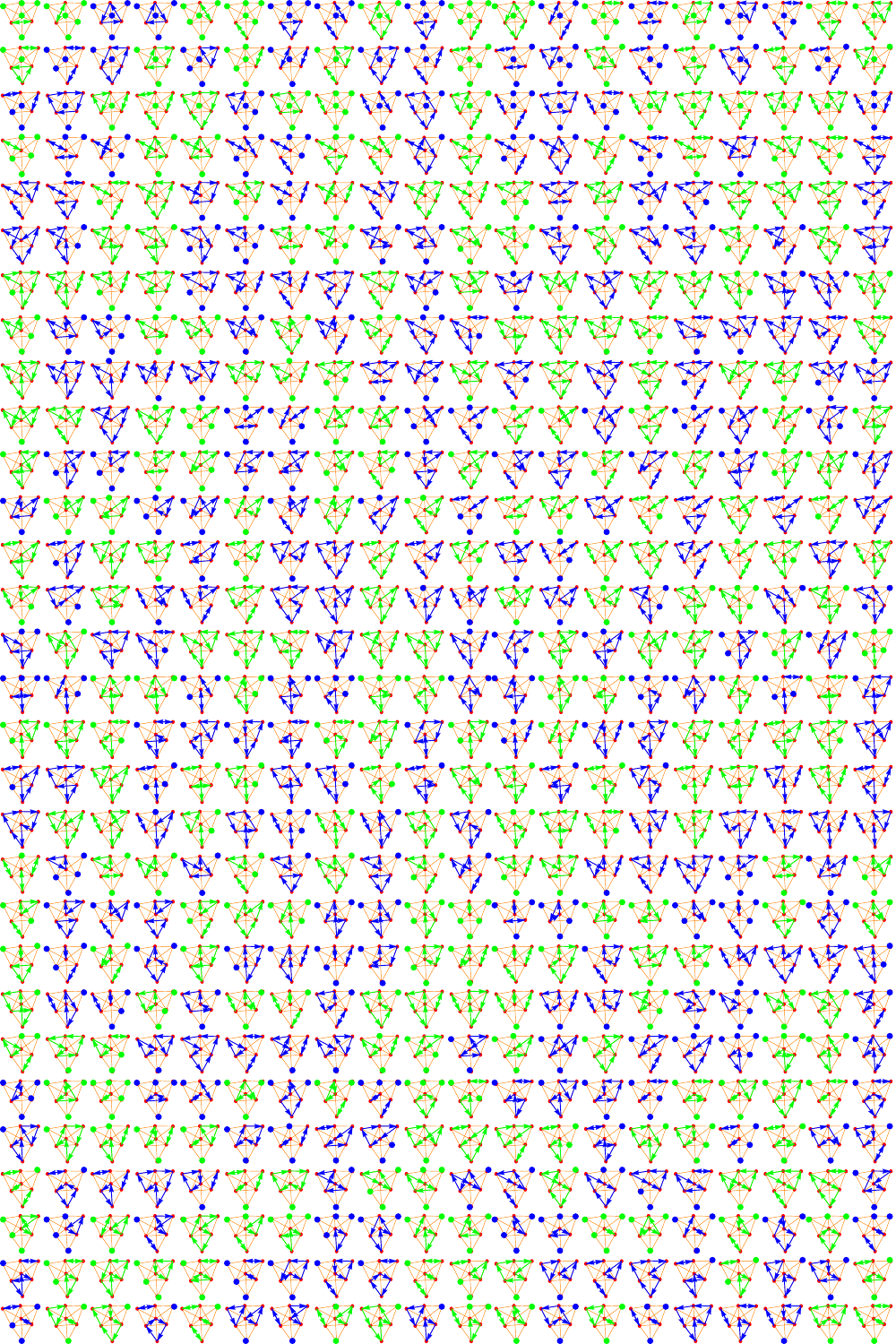}}
\caption{
There are $601$ one-dimensional oriented subgraphs of
$G'=K_3'$ if $G=K_3$ is equipped with the Whitney complex. 
We see the $600$ paths different from the identity.
$300$ permutations have signature $1$, and $301$ permutations 
have signature $-1$. We have $\psi(G)=300-301=-1$ 
and $\phi(G)=(-1)^3=-1$. 
\label{figure6}
}
\end{figure}

\vfill
\clearpage

\section{Some code}

The following code is written in the Mathematica language
which knows graphs as a fundamental data structure. 
We first compute the Fredholm
determinant $\det(1+A)$ of a graph $G$ with adjacency matrix $A$.
The procedure ``Whitney" produces the Whitney complex of a graph, the set of all
complete subgraphs which is a finite abstract simplicial complex. 
This procedure is then is used to get the connection graph of a graph.
We then compute $\psi(G)$ and $\phi(G)$ for 100 random graphs. 

\begin{tiny}
\lstset{language=Mathematica} \lstset{frameround=fttt}
\begin{lstlisting}[frame=single]
Fredholm[A_]:=A+IdentityMatrix[Length[A]];
FredholmDet[s_]:=Det[Fredholm[AdjacencyMatrix[s]]];
Experiment1=Table[FredholmDet[WheelGraph[k]],{k,4,20}]

CliqueNumber[s_]:=Length[First[FindClique[s]]];
ListCliques[s_,k_]:=Module[{n,t,m,u,r,V,W,U,l={},L},L=Length;
  VL=VertexList;EL=EdgeList;V=VL[s];W=EL[s]; m=L[W]; n=L[V];
  r=Subsets[V,{k,k}];U=Table[{W[[j,1]],W[[j,2]]},{j,L[W]}];
  If[k==1,l=V,If[k==2,l=U,Do[t=Subgraph[s,r[[j]]];
  If[L[EL[t]]==k(k-1)/2,l=Append[l,VL[t]]],{j,L[r]}]]];l];
Whitney[s_]:=Module[{F,a,u,v,d,V,LC,L=Length},V=VertexList[s];
  d=If[L[V]==0,-1,CliqueNumber[s]];LC=ListCliques;
  If[d>=0,a[x_]:=Table[{x[[k]]},{k,L[x]}];
  F[t_,l_]:=If[l==1,a[LC[t,1]],If[l==0,{},LC[t,l]]];
  u=Delete[Union[Table[F[s,l],{l,0,d}]],1]; v={};
  Do[Do[v=Append[v,u[[m,l]]],{l,L[u[[m]]]}],{m,L[u]}],v={}];v];
ConnectionGraph[s_] := Module[{c=Whitney[s],n,A},n=Length[c];
   A=Table[1,{n},{n}];Do[If[DisjointQ[c[[k]],c[[l]]]||
   c[[k]]==c[[l]],A[[k,l]]=0],{k,n},{l,n}];AdjacencyGraph[A]];
Fvector[s_] := Delete[BinCounts[Length /@ Whitney[s]], 1];
FermiCharacteristic[s_]:=Module[{f=Fvector[s]},
                   (-1)^Sum[f[[2k]],{k,Floor[Length[f]/2]}]];
FredholmCharacteristic[s_]:=FredholmDet[ConnectionGraph[s]];

Experiment2 = Do[s=RandomGraph[{10,30}];
  Print[{FermiCharacteristic[s],FredholmCharacteristic[s]}],
{100}];
\end{lstlisting}
\end{tiny}

Lets look at the proposition

\begin{tiny}
\lstset{language=Mathematica} \lstset{frameround=fttt}
\begin{lstlisting}[frame=single]
ConnectionGraph[s_,V_]:=Module[{c=Whitney[s],n,A},n=Length[c];
  A=Table[1,{n+1},{n+1}];Do[If[DisjointQ[c[[k]],c[[l]]] ||
  c[[k]]==c[[l]],A[[k,l]]=0],{k,n},{l,n}]; A[[n+1,n+1]]=0;
  Do[If[DisjointQ[V,c[[k]]],A[[k,n+1]]=0; A[[n+1,k]]=0],{k,n}];
  AdjacencyGraph[A]];
EulerChi[s_]:=Module[{c=Whitney[s]},
  Sum[(-1)^(Length[c[[k]]]-1),{k,Length[c]}]];
Experiment3 = Do[ s=RandomGraph[{30,50}]; v=VertexList[s];
   V=RandomChoice[v,Random[Integer,Length[v]]]; h=Subgraph[s,V];
   ss=ConnectionGraph[s]; sss=ConnectionGraph[s,V];
   u={FredholmDet[ss],FredholmDet[sss],1-EulerChi[h]};
   Print[u," ",u[[1]] u[[3]]==u[[2]]],{100} ];
\end{lstlisting}
\end{tiny}

And here is an other experiment. What is the probability of a random
Erd\H{o}s R\'enyi graph in $E(n,p)$ to have connection Fredholm
determinant $1$? In this example, we experiment with $n=12$ and $p=0.6$:

\begin{tiny}
\lstset{language=Mathematica} \lstset{frameround=fttt}
\begin{lstlisting}[frame=single]
ErdoesRenyi[n_,p_]:=RandomGraph[{n,Floor[p n(n-1)/2]}];
G[n_,p_]:=FredholmDet[ConnectionGraph[ErdoesRenyi[n,p]]];
n=12;p=0.6;k=0.;m=0;
Experiment4 = Do[m++;If[G[n,p]==1,k++];Print[k/m],{Infinity}];
\end{lstlisting}
\end{tiny}

And here are the procedures producing 
the prime graphs $G_n$ and prime connection graph $H_n$.
In the first case we verify numerically $\chi(G_n) = \sum_k (-\mu(k))$, 
in the second case we verify numerically $\psi(H_n) = \prod_k (-\mu(k))$, 
where $\mu(k)$ is the M\"obius function of the integer $k$. 

\begin{tiny}
\lstset{language=Mathematica} \lstset{frameround=fttt}
\begin{lstlisting}[frame=single]
PrimeGraph[M_]:=Module[{V={},e,s},
 Do[If[MoebiusMu[k]!=0,V=Append[V,k]], {k,2,M}];e={};
 Do[If[(Divisible[V[[k]],V[[l]]] || Divisible[V[[l]],V[[k]]]),
 e=Append[e,V[[k]]->V[[l]]]],{k,Length[V]},{l,k+1,Length[V]}];
 UndirectedGraph[Graph[V,e]]];

PrimeConnectionGraph[M_]:=Module[{V={},e,s},
 Do[If[MoebiusMu[k]!=0,V=Append[V,k]], {k,2,M}];e={};
 Do[If[GCD[V[[k]],V[[l]]]>1 || GCD[V[[l]],V[[k]]]>1,
 e=Append[e,V[[k]]->V[[l]]]],{k,Length[V]},{l,k+1,Length[V]}];
 UndirectedGraph[Graph[V,e]]];

Test5[M_]:=Module[{V,g,h,j}, h=PrimeConnectionGraph[M]; 
 g=PrimeGraph[M]; V=VertexList[g];
 j=Table[-MoebiusMu[V[[k]]],{k,Length[V]}]; 
 Print[{Total[j]                     ,EulerChi[g],
        Product[j[[k]],{k,Length[j]}],FredholmDet[h]}]];
Experiment5 = Do[Test5[k],{k,10,200}];
\end{lstlisting}
\end{tiny}

\bibliographystyle{plain}

\end{document}